\documentclass[10pt]{article}

\usepackage{amsmath,amsthm,amssymb}
\usepackage{caption,graphicx,subfig}
\usepackage[text={6.5in,9.25in},centering]{geometry}
\usepackage[sort,numbers]{natbib}
\usepackage{xcolor}

\setcaptionmargin{0.25in}

\makeatletter\@addtoreset{equation}{section}\makeatother

\newtheorem{thm}{Theorem}[section] 
\newtheorem{lem}[thm]{Lemma}  
 
\newtheorem{prop}[thm]{Proposition}  
\newtheorem{hyp}{Hypothesis}

\newtheoremstyle{named}{}{}{\itshape}{}{\bfseries}{.}{.5em}{\thmnote{#3's }#1}
\theoremstyle{named}

\theoremstyle{definition}

\begin{document}

\title{Isolas of multi-pulse solutions to lattice dynamical systems}

\author{
Jason J. Bramburger\thanks{Division of Applied Mathematics, Brown University, Providence, RI, 02906}\ \thanks{Department of Mathematics and Statistics, University of Victoria, Victoria, BC, V8P 5C2}
}

\date{}
\maketitle

\begin{abstract}
This work investigates the existence and bifurcation structure of multi-pulse steady-state solutions to bistable lattice dynamical systems. Such solutions are characterized by multiple compact disconnected regions where the solution resembles one of the bistable states and resembles another trivial bistable state outside of these compact sets. It is shown that the bifurcation curves of these multi-pulse solutions lie along closed and bounded curves (isolas), even when single-pulse solutions lie along unbounded curves. These results are applied to a discrete Nagumo differential equation and we show that the hypotheses of this work can be confirmed analytically near the anti-continuum limit. Results are demonstrated with a number of numerical investigations.      
\end{abstract}

%%%%%%%%%%%%%%%%%%%%%%%%%%%%%%%%%%%%%%%%%%%%%%%%%%%%%%%%%%%%%%%%%%%%%%%%%%%%

\section{Introduction} %Section: Introduction ---------------------------------------------------------------------------------------------------------------------------------------------------------

The competition between bistable states in nonlinear systems can lead to fascinating and unintuitive structures. Some of the most documented examples are localized structures which resemble a patterned or activated state inside of a compact spatial region and a second homogeneous state outside of this compact region. Localized structures can be found in many applications, including as crime hotspots \cite{Hotspot,Hotspot2,Hotspot3}, vegetation patterns \cite{Veg1,Veg2,Veg3}, and soft matter quasicrystals \cite{Subramanian}. They have further been observed in chemical reactions \cite{Chemical}, supported elastic struts \cite{Michaels}, semiconductors \cite{Semiconductor}, and ferrofluids \cite{Ferrofluid}.   

\begin{figure} %Snaking Figure
\centering
\includegraphics[width=\textwidth]{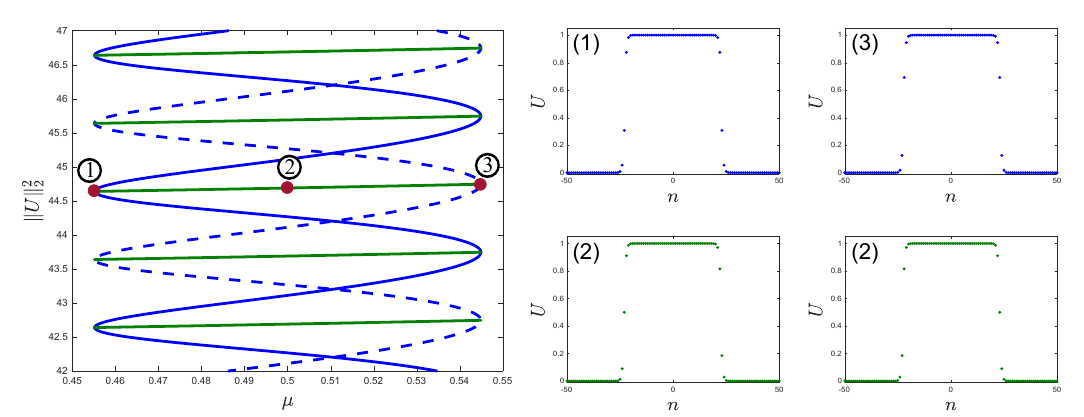}
\caption{Snaking of single-pulses in (\ref{LDS_Intro}) with $d = 0.1$. Symmetric pulses come in two types: on-site (solid blue) which roughly have an odd number of elements on their plateau and off-site (dashed blue) which roughly have an even number of elements on their plateau. The two bifurcation curves of symmetric equilibria are connected by branches of asymmetric solutions which bifurcate near the left and right extremities of the symmetric curves in pitchfork bifurcations. Asymmetric solutions come in pairs, as is demonstrated by the sample profiles $(2)$.}
\label{fig:Snaking}
\end{figure} 

In this manuscript we focus on localized solutions to lattice dynamical systems. As an example, consider the discrete Nagumo equation  
	\begin{equation}\label{LDS_Intro}
		\dot{U}_n = d(U_{n+1} + U_{n-1} - 2U_n) + U_n(U_n - \mu)(1-U_n), \quad n\in\mathbb{Z},
	\end{equation}
where $d \geq 0$ describes the strength of interaction between nearest-neighbours on the lattice $\mathbb{Z}$, $\mu$ is a real bifurcation parameter, and $U_n$ are the real-valued state variables. As one can see in Figure~\ref{fig:Snaking}, localized steady-state solutions of (\ref{LDS_Intro}) can arrange themselves in complicated existence diagrams with respect to varying $\mu$. One can see that the solutions have a single connected region of activation with $U_n \approx 1$, while outside of this region of activation the solution resembles the trivial rest state in that $U_n \approx 0$. This single connected region of activation leads to the terminology that these localized solutions are single-pulses. Furthermore, symmetric single-pulses of (\ref{LDS_Intro}) lie along unbounded curves that bounce back and forth between fixed values of $\mu$, while the length of the region of activation monotonically increases without bound. Such a bifurcation scenario is termed {\em snaking} and it has been documented extensively in lattice dynamical systems \cite{Bramburger,Chong,Chong2,Kusdiantara,McCullen,Papangelo,Susanto,Taylor,Yulin}. Beyond the symmetric single-pulses, there also exist asymmetric single-pulses which bifurcate from the symmetric snaking branches in a pitchfork bifurcation. The bifurcation curves of asymmetric single-pulses are known to connect the two different snaking branches with endpoints given by pitchfork bifurcations which take place near turning points of opposite curvature on the symmetric branches.         

In the case of partial differential equations posed on $\mathbb{R}$, complete analytical descriptions of the processes that lead to snaking are now available \cite{Aougab,Beck,kozyreff1,kozyreff2}. In particular, it has been shown that the specific form of the bifurcation curves of single-pulse solutions are entirely dictated by the bifurcation structure of front solutions which asymptotically connect the homogeneous background state to the patterned or activated state. Recently these results were extended to lattice dynamical systems in \cite{Bramburger} and fully explain the organization of bifurcation curves in Figure~\ref{fig:Snaking}. Furthermore, in the case of (\ref{LDS_Intro}) we can exploit the anti-continuum limit, corresponding to the uncoupled system arising when setting $d = 0$, to verify the conditions of the general theory for $0 < d \ll 1$. This rigorous verification of the theoretical analysis is something which is at present completely unavailable in the continuous spatial setting.

Following \cite{Bramburger}, steady-state solutions of (\ref{LDS_Intro}) are bounded solutions of the discrete dynamical system
\begin{equation}\label{Map_Intro}
	\begin{split}
		u_{n+1} &= v_n, \\
		v_{n+1} &= 2v_n - u_n -\frac{1}{d}v_n(v_n - \mu)(1 - v_n),
	\end{split}
\end{equation}
where $(u_n,v_n)=(U_{n-1},U_n)$. In the setting of (\ref{Map_Intro}), fronts and localized solutions of (\ref{LDS_Intro}) manifest themselves as heteroclinic and homoclinic orbits, respectively. More precisely, the single-pulse solutions presented in Figure~\ref{fig:Snaking} are homoclinic orbits to the trivial fixed point $(u_n,v_n) = (0,0)$ for which a large number of iterates remain in a neighbourhood of another fixed point $(u_n,v_n) = (1,1)$, corresponding to the homogeneous steady-state $U_n = 1$ of (\ref{LDS_Intro}). In the case of single-pulse solutions, these iterates only enter and leave the neighbourhood of the fixed point $(u_n,v_n) =(1,1)$ once, but Figure~\ref{fig:2Pulse} presents evidence that there exists solutions which enter and leave this neighbourhood at least twice as well. Solutions that enter and leave this neighbourhood more than once are termed multi-pulses since the corresponding localized steady-states of (\ref{LDS_Intro}) have multiple disconnected regions of activation. More precisely, a $k$-pulse is a multi-pulse that has exactly $k \geq 2$ disconnected regions of activation. Figure~\ref{fig:2Pulse} further demonstrates that although the single-pulses of (\ref{LDS_Intro}) snake, at least some of the multi-pulses do not. That is, Figure~\ref{fig:2Pulse} presents an {\em isola} of $2$-pulses, a closed curve in the bifurcation diagram. 

\begin{figure} %2Pulse Isola Figure
\centering
\includegraphics[width=\textwidth]{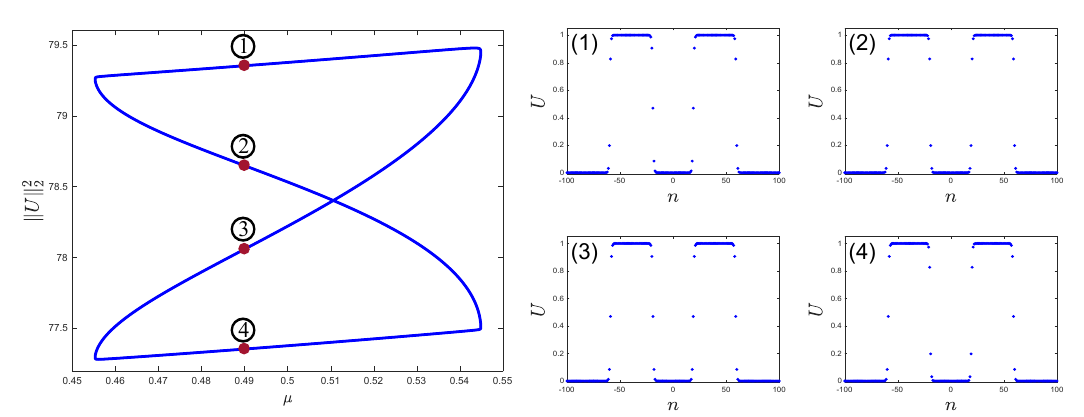}
\caption{An isola of symmetric 2-pulses in (\ref{LDS_Intro}) with $d = 0.1$. Sample profiles $(1)$ and $(4)$ resemble two mirrored asymmetric single-pulses, whereas $(2)$ and $(3)$ resemble two symmetric single-pulses. All sample profiles are provided for the parameter value $\mu = 0.49$.}
\label{fig:2Pulse}
\end{figure} 

A number of numerical investigations have shown that multi-pulses lie along isolas \cite{Burke,Heijden,2Pulse,Wadee}, leading to the conjecture that this is always the case. Despite the significant amount of attention on the bifurcation structure of single-pulse solutions to both lattice dynamical systems and partial differential equations, only the work of Knobloch et al. \cite{2Pulse} has provided positive affirmation of this conjecture for $2$-pulse solutions to partial differential equations. In this manuscript we extend this result to show that multi-pulse solutions of a class of lattice dynamical systems with an arbitrary number of disconnected regions of localization lie along isolas. In particular, this work applies to (\ref{LDS_Intro}), demonstrating the existence and bifurcation structure of a number of localized solutions to lattice dynamical systems. Hence, this work goes far beyond the results known for the spatially continuous setting of partial differential equations and therefore its techniques could be used to inform future studies of localized structures beyond the lattice setting.    

The existence of multi-pulse solutions to (\ref{LDS_Intro}) should not come as a surprise to the reader. Indeed, seminal results in the theory of discrete dynamical systems such as the Smale horseshoe \cite{Smale} and the $\lambda$-Lemma \cite{Palis} can be used to demonstrate the existence of multi-pulses based on the existence of transverse homoclinic orbits of (\ref{Map_Intro}), i.e. single-pulses of (\ref{LDS_Intro}). What is new to this work is that we describe the complete bifurcation structure of multi-pulses to show that they cannot snake, even when the single-pulses do. Furthermore, instead of using homoclinic orbits to obtain the existence of multi-pulses, this work and its predecessor \cite{Bramburger} establish their results using a single curve of heteroclinic orbits. As mentioned above, these heteroclinic orbits correspond to steady-state fronts of the lattice equation (\ref{LDS_Intro}), which have long been studied in the context of traveling wave solutions to lattice dynamical systems which fail to propagate \cite{Brucal,Clerc,Elmer,Elmer2,Fath,Guo,Keener}. Therefore, this work builds off of these previous studies since their results can be used to confirm the hypotheses required for the results of this manuscript.      

This paper is organized as follows. In Section~\ref{sec:MainResults} we formulate the hypotheses and present the main results for general reversible discrete dynamical systems. We then turn back to the specific example of the discrete Nagumo equation (\ref{LDS_Intro}) in Section~\ref{sec:LDS}. Our discussion of equation \eqref{LDS_Intro} includes an analytic verification of the hypotheses for $0 < d \ll 1$ and numerical validation of the the results, followed by a brief discussion of the expected stability properties of the single- and multi-pulses. We leave all proofs to Section~\ref{sec:Proofs} and conclude with a discussion of the results and future directions in Section~\ref{sec:Discussion}.

\section{Main Results}\label{sec:MainResults} %Section: Main Results ---------------------------------------------------------------------------------------------------------------------------------------------------------

We consider a smooth function $F:\mathbb{R}^2 \times \mathbb{R} \to \mathbb{R}^2$ to define the mapping
\begin{equation}\label{Map}
	u_{n+1} = F(u_n,\mu),
\end{equation}
where $\mu \in \mathbb{R}$ is a bifurcation parameter. We further assume that $F$ is a diffeomorphism for each fixed $\mu$, leading to the backwards iteration scheme
\begin{equation}\label{InverseMap}
	u_{n-1} = F^{-1}(u_n,\mu),
\end{equation} 
where $F^{-1}$ is the inverse of $F$ at each $\mu\in\mathbb{R}$. The following hypothesis assumes that $F$ is a reversible mapping.

\begin{hyp}\label{hyp:Reverser} %Hypothesis: Reversibility
	There exists a linear map $\mathcal{R}:\mathbb{R}^2\to\mathbb{R}^2$ with $\mathcal{R}^2 = 1$ and $\mathrm{dim\ Fix}(\mathcal{R}) = 1$ so that $F^{-1}(u,\mu) = \mathcal{R}F(\mathcal{R}u,\mu)$ for all $u \in \mathbb{R}^2$ and $\mu \in\mathbb{R}$.
\end{hyp} %

Hypothesis~\ref{hyp:Reverser} implies that if $\{u_n\}_{n\in\mathbb{Z}}$ is a solution of (\ref{Map}), so is $\{\mathcal{R}u_{-n}\}_{n\in\mathbb{Z}}$. Then, a solution $\{u_n\}_{n\in\mathbb{Z}}$ of (\ref{Map}) is said to be {\bf symmetric} if $\mathcal{R}\{u_n\}_{n\in\mathbb{Z}} = \{u_n\}_{n\in\mathbb{Z}}$. In \cite[Lemma~2.1]{Bramburger} it was shown that a solution $\{u_n\}_{n\in\mathbb{Z}}$  to (\ref{Map}) is symmetric if, and only if, there exists an $n\in\mathbb{Z}$ such that $\mathcal{R}u_n = u_n$ or $\mathcal{R}u_{n-1} = u_n$. We refer to symmetric solutions satisfying $\mathcal{R}u_n = u_n$ for some $n \in \mathbb{Z}$ as {\bf on-site}, while those satisfying $\mathcal{R}u_{n-1} = u_n$ for some $n \in \mathbb{Z}$ are referred to as {\bf off-site}.   

\begin{hyp}\label{hyp:FixedPts} %Hypothesis: Fixed Points
	There exists a compact interval $J \subset \mathbb{R}^2$ with nonempty interior such that for each $\mu \in J$, the points $u = 0,u_* \in {\rm Fix}(\mathcal{R})$ are hyperbolic fixed points of (\ref{Map}). We further assume that the eigenvalues of the matrices $F_u(0,\mu)$ and $F_u(u_*,\mu)$ are real and positive for all $\mu \in J$.
\end{hyp} %

We note that the assumption that the eigenvalues of the matrices $F_u(0,\mu)$ and $F_u(u_*,\mu)$ be positive is not a major restriction to our work here. The reason for this is that reversibility of (\ref{Map}) enforces a strict structure to the eigenvalues of the matrices $F_u(0,\mu)$ and $F_u(u_*,\mu)$ in that if $\lambda$ is a nonzero eigenvalue, then so must be $\bar{\lambda}, \lambda^{-1}, \bar{\lambda}^{-1}$ \cite[Proposition~16.3.4]{Wiggins}. Hence, knowing that $u=0,u_*$ are hyperbolic implies that both eigenvalues of the matrices $F_u(0,\mu)$ and $F_u(u_*,\mu)$ are real and have the same sign. In the case that they are both negative one may consider the second-iterate map, $F^2:=F\circ F$, in place of $F$ to guarantee that both Hypotheses~\ref{hyp:Reverser} and \ref{hyp:FixedPts} are satisfied.

Based on the previous comments, we see that for all $\mu \in J$ the fixed points $u = 0,u_*$ must be saddles. Therefore, the stable manifold theorem implies that for all $\mu \in J$, both $u = 0$ and $u = u_*$ have one-dimensional stable and unstable manifolds associated to them. Throughout this manuscript we will denote the stable and unstable manifolds of the fixed point $u = 0$ as $W^s(0,\mu)$ and $W^u(0,\mu)$, respectively. Analogously, $W^s(u_*,\mu)$ and $W^u(u_*,\mu)$ denote the stable and unstable manifolds of $u_*$, and we note all stable and unstable manifolds are smooth with respect to varying $\mu \in J$. Reversibility of (\ref{Map}) guarantees the following identities
\begin{equation}\label{RevMans}
	W^s(0,\mu) = \mathcal{R}W^u(0,\mu), \quad \quad W^s(u_*,\mu) = \mathcal{R}W^u(u_*,\mu)
\end{equation}
for all $\mu \in J$. 

Our interest now lies in characterizing heteroclinic orbits of (\ref{Map}) which connect $0$ and $u_*$ asymptotically. Particularly, the set 
\[
	X := \bigcup_{\mu \in J} (W^u(0,\mu)\cap W^s(u_*,\mu))\times\{\mu\} \subset \ell^\infty \times J
\] 
is the set of all heteroclinic orbits from the fixed point $0$ to $u_*$ over all $\mu \in J$. Using (\ref{RevMans}) we can see that 
\[
	\mathcal{R}X := \{(\mathcal{R}u,\mu):\ (u,\mu)\in X\},
\] 
is the set of all heteroclinic orbits from the fixed point $u_*$ to $0$ over all $\mu \in J$. We note that the elements of $X$ and $\mathcal{R}X$ are uniformly bounded, and hence they can naturally be seen as subsets of the Banach space $\ell^\infty \times J$, where $\ell^\infty$ is the set of all uniformly bounded sequences indexed by $\mathbb{Z}$ with norm given by
\[
	\|u\|_\infty := \sup_{n\in\mathbb{Z}} |u_n|,
\]    
for all $u = \{u_n\}_{n\in\mathbb{Z}} \in \ell^\infty$. The fact that (\ref{Map}) is autonomous implies that it is equivariant with respect to the left shift operator, $S:\ell^\infty\to\ell^\infty$, acting by
\begin{equation}\label{Shift}
	[Su]_n := u_{n+1}, \quad \forall n\in\mathbb{Z},u\in\ell^\infty.
\end{equation}
This equivariance property implies that if $u$ is a solution of (\ref{Map}), then so must be $Su$. Hence, to identify whether orbits of (\ref{Map}) are simply shifts of each other or not we consider the orbit space $\ell^\infty/\langle S\rangle$, which is the set of equivalence classes in $\ell^\infty$ with respect to the shift $S$ given such that for $u,v\in\ell^\infty$ we have $u\sim v$ if, and only if, there exists $p \in \mathbb{Z}$ such that $S^pu = v$. We will write $[u] = \{v\in\ell^\infty:\ u\sim v\}$ to denote the equivalence class of an element $u$ and define the quotient mapping 
\[
	\pi:\ell^\infty \times J \to \ell^\infty/\langle S\rangle \times J, \quad (u,\mu) \mapsto ([u],\mu)
\]
onto the orbit space. This leads to our final hypothesis. 

\begin{hyp}\label{hyp:Heteroclinic} %Hypothesis: Gamma set
	There exists a smooth, connected curve $\Gamma \subset X$ satisfying the following:
	\begin{enumerate}
		\item We have $\Gamma \cap (\ell^\infty\times \partial J) =\emptyset$, and there exists $K > 0$ such that $\|u\|_\infty \leq K$ for all $(u,\mu)\in \Gamma$.
		\item If $(u,\mu) \in \Gamma$ then the elements $u$ lie either along a transverse intersection or a quadratic tangency of the manifolds $W^u(0,\mu)$ and $W^s(u_*,\mu)$.  
		\item The set $\bar{\Gamma} := \pi(\Gamma)$ is a closed loop. That is, we can parameterize $\bar{\Gamma}$ by a smooth map $\gamma:[0,1]\to\bar{\Gamma}$ by $s \mapsto ([u](s),\mu(s))$ with $\gamma(0) = \gamma(1)$.
	\end{enumerate}
\end{hyp} %

Hypothesis~\ref{hyp:Heteroclinic} contains all of our assumptions about the existence of a smooth curve of heteroclinic orbits of (\ref{Map}) with respect to varying $\mu \in J$. We begin by assuming that $\Gamma$ is a smooth curve embedded in the interior of $\ell^\infty \times J$. Our second assumption dictates that the manifolds $W^u(0,\mu)$ and $W^s(u_*,\mu)$ intersect in the simplest ways possible. Generically these types of intersections should be all that are expected, although it may be possible to weaken Hypothesis~\ref{hyp:Heteroclinic} to more exotic intersections and still obtain the same results of this manuscript. The third and final assumption states that tracing out the curve $\Gamma$ in $\ell^\infty \times J$ eventually either returns to where it started, or to a shift of the original heteroclinic orbit for the same value of $\mu$.  

Our interest lies in constructing homoclinic orbits of the trivial fixed point to (\ref{Map}) that enter and leave a neighbourhood of the fixed point $u_*$ exactly $k \geq 2$ times. We are not only interested in determining the existence of such homoclinic orbits, hereby referred to {\em $k$-pulses}, but to understand their behaviour as $\mu$ is varied throughout $J$. The final assumption in Hypothesis~\ref{hyp:Heteroclinic} plays a crucial role in performing this task. To see this, notice that for each $s \in [0,1]$, each element of the curve $\gamma(s)$ lifts to infinitely many points in $\Gamma$, all of which are merely shifts of each other. Identifying one such preimage of $\gamma(0)$, we may produce a smooth curve in $\Gamma$, written as $(u(s),\mu(s))$, for which $\pi(u(s),\mu(s)) = \gamma(s)$ for all $s \in [0,1]$. Then, Hypothesis~\ref{hyp:Heteroclinic} states that $u(1) = S^pu(0)$ for some $p \in \mathbb{Z}$. It was shown in the  preceding work \cite{Bramburger} that the value of $p$ plays an important role in determining the bifurcation structure of $1$-pulses: if $p = 0$ the bifurcation curves are isolas and if $p \neq 0$ then the bifurcation curves snake, as in Figure~\ref{fig:Snaking}. Our present analysis still requires the assumption that $\pi(\Gamma)$ is a closed loop, but now we prove the existence and bifurcation structure of $k$-pulses with $k \geq 2$. The following theorem is our main result and particularly shows that all $k$-pulses lie along isolas, regardless of the value of $p$.

\begin{thm}\label{thm:MainResult} %Theorem: Main result on k-pulses
	For each $k \geq 2$, there exists $M_k \gg 1$, so that for each set of integers $N_1,\dots,N_{2k-1} \geq M_k$ there exists a $k$-pulse solution of (\ref{Map}). Generically the following is true:
	\begin{enumerate}
		\item (Symmetric Pulses) There exists both symmetric and asymmetric $k$-pulses of (\ref{Map}) with $N_j = N_{2k-j}$ for all $j \in \{1,\dots,2k-1\}$. The symmetric $k$-pulses lie along smooth closed curves in $\ell^\infty \times J$, are on-site if $N_k$ is odd, and off-site otherwise. 
		\item (Pitchforks to Asymmetric Pulses) Let $N_0 = \min\{N_1,\dots,N_{2k-1}\}$. There exists $\eta \in (0,1)$ such that if $\mu_\mathrm{sn} \in \mathring{J}$ is the location of a saddle-node bifurcation on the symmetric $k$-pulse curve, then there exists $\mu_\mathrm{pf} \in \mathring{J}$ with $|\mu_\mathrm{sn} - \mu_\mathrm{pf}| = \mathcal{O}(\eta^{N_0})$ and the property that at $\mu = \mu_\mathrm{pf}$ two branches of asymmetric  orbits (mapped into each other by $\mathcal{R}$) emanate in a pitchfork bifurcation from the symmetric $k$-pulse curve. The resulting bifurcation curves in $\ell^\infty\times J$ of these asymmetric $k$-pulses are smooth having endpoints given by pitchfork bifurcations near saddle-node bifurcations of the symmetric $k$-pulses with opposite curvature. 
		\item (Asymmetric Pulses) If $N_j \neq N_{2k-j}$ for at least one $j \in \{1,\dots,2k-1\}$ then the resulting $k$-pulse is asymmetric and lies along a smooth closed curve in $\ell^\infty\times J$. This curve of asymmetric $k$-pulses does not exhibit any other bifurcations than saddle-nodes.  
	\end{enumerate}
\end{thm} %End of theorem

We leave the proof of Theorem~\ref{thm:MainResult} to Section~\ref{sec:Proofs} where the reader can find an array of auxiliary results which help to clarify the notation in the statement of the theorem. Here we note that roughly the integers $N_1,N_2,\dots, N_{2k-1}$ describe the number of iterates the orbit spends in a neighbourhood of $0$ and $u_*$. More precisely, $N_j$ with $j$ odd describes the number of iterates spent close to $u_*$, whereas $N_j$ with $j$ even describes the number of iterates spent close to $0$ between iterates close to $u_*$. Figure~\ref{fig:kPulse} provides a visual description of this. 

\begin{figure} %Anatomy of a k-Pulse
\centering
\includegraphics[width=0.8\textwidth]{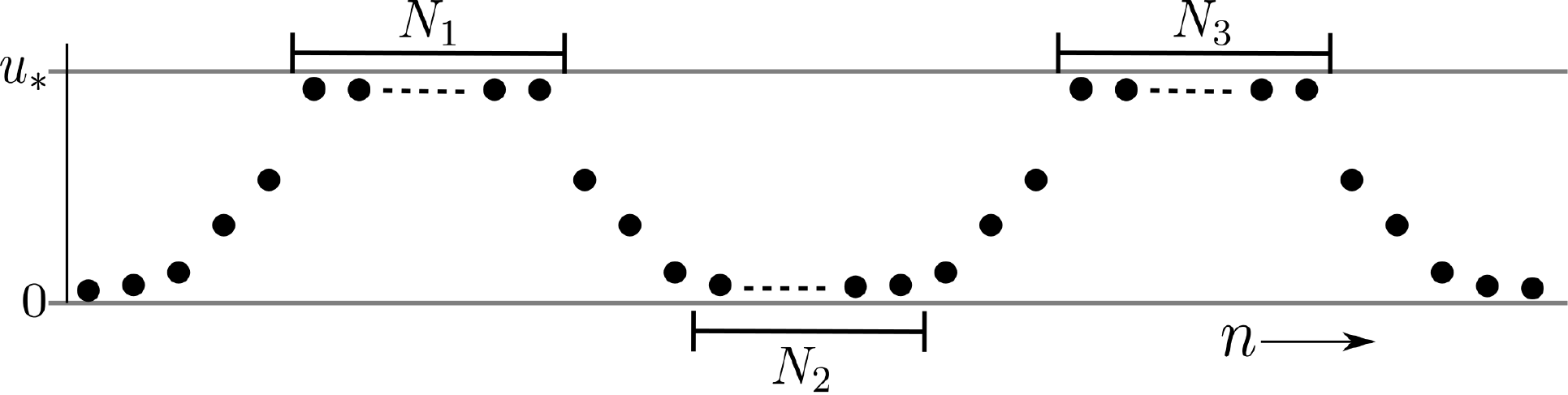}
\caption{A visual depiction of the results of Theorem~\ref{thm:MainResult}. We construct homoclinic orbits of (\ref{Map}) which for sufficiently large integers $N_1,N_2,\dots,N_{2k-1}$, the homoclinic orbit spends $N_1$ iterates in a neighbourhood of $u_*$, then jumps to a neighbourhood of $0$ for $N_2$ iterates, then back to a neighbourhood of $u_*$ for $N_3$ iterates, and so on for any finite sequence of sufficiently large $N_j$.}
\label{fig:kPulse}
\end{figure}

\section{Application to Lattice Dynamical Systems}\label{sec:LDS} %Section: LDSs --------------------------------------------------------------------------------------------------------------------------------- 

We now return to the discrete Nagumo system discussed in the introduction, given by
\begin{equation}\label{LDS}
	\dot{U}_n = d(U_{n+1} + U_{n-1} - 2U_n) + U_n(U_n - \mu)(1-U_n), \quad n\in\mathbb{Z}.
\end{equation}
Throughout this section we will demonstrate how our theoretical results in Section~\ref{sec:MainResults} can be applied to determine the existence and bifurcation structure of localized solutions with multiple disconnected regions of activation. The parameter $d > 0$ represents the strength of coupling between neighbouring elements indexed by the integer lattice $\mathbb{Z}$. We take $\mu \in [0,1]$ to be a bifurcation parameter, and note that in the interior of this parameter region we have exactly two stable spatially homogeneous steady-state solutions of (\ref{LDS}) given by $U_n = 0,1$ for all $n \in \mathbb{Z}$ and one unstable steady-state given by $U_n = \mu$ for all $n \in \mathbb{Z}$. At the endpoint $\mu = 0$ a transcritical bifurcation takes place when the equilibria $U_n = 0$ and $U_n = \mu$ collide, and similarly another transcritical bifurcation takes place at $\mu = 1$ when $U_n = \mu$ and $U_n = 1$ collide.    

As detailed in the introduction, searching for nontrivial steady-states of (\ref{LDS}) requires solving the infinite systems of equations 
\begin{equation}\label{LDS_Steady}
	0 = d(U_{n+1} + U_{n-1} - 2U_n) + U_n(U_n - \mu)(1-U_n), \quad n\in\mathbb{Z},	
\end{equation}
obtained by setting $\dot{U}_n = 0$ for all $n \in \mathbb{Z}$. As one can now see, system (\ref{LDS_Steady}) defines a delayed discrete dynamical systems in the spatial index $n$. We may introduce the change of variable $(u_n,v_n) = (U_{n-1},U_n)$ for all $n \in \mathbb{Z}$ to obtain the first-order mapping 
 \begin{equation}\label{LDS_Map}
 	\begin{split}
		u_{n+1} &= v_n, \\
		v_{n+1} &= 2v_n - u_n - \frac{1}{d}[v_n(v_n - \mu)(1-v_n)],
	\end{split}
 \end{equation}
which is a diffeomorphism of the form (\ref{Map}) studied in this work. Note that bounded solutions of (\ref{LDS_Map}) correspond to steady-state solutions of (\ref{LDS}), and most importantly, trajectories of (\ref{LDS_Map}) which are homoclinic to the trivial fixed point $(0,0)$ are localized steady-states of (\ref{LDS}). Additionally, the symmetry of the coupling terms in (\ref{LDS}) endows (\ref{LDS_Map}) with a reversible structure with reverser given in matrix form as
\[
	\mathcal{R} = \begin{bmatrix}
		0 & 1\\ 1 & 0
	\end{bmatrix},
\]	 
thus satisfying Hypothesis~\ref{hyp:Reverser}. The spatially independent steady-states of (\ref{LDS}) manifest themselves as fixed-points of (\ref{LDS_Map}) belonging to $\mathrm{Fix}(\mathcal{R})$. For any $d > 0$ the fixed points $(0,0)$ and $(1,1)$ are hyperbolic for all $\mu \in (0,1)$ and can be shown to satisfy Hypothesis~\ref{hyp:FixedPts} for any closed interval $J \subset (0,1)$.

%\subsection{Flat Plateaus}\label{subsec:Flat} %Subsection: Flat ---------------------------------------------------------------------------------------------------------------------------------

{\bf Verifying Hypothesis~\ref{hyp:Heteroclinic}.} Let us demonstrate an application of the results of Theorem~\ref{thm:MainResult} to the mapping (\ref{LDS_Map}). Throughout we will take $u_* = (1,1)$, and so the desired homoclinic orbits of (\ref{LDS_Map}) that spend a long time near the fixed point $(1,1)$ represent localized steady-state solutions of (\ref{LDS}) which resemble the spatially homogeneous steady-state $U_n = 1$ on some compact subset of the indices. We refer the reader to Figures~\ref{fig:Snaking} and \ref{fig:2Pulse} for characteristic examples of such solutions.  

First, it should be noted that confirming Hypothesis~\ref{hyp:Heteroclinic} can potentially be a difficult task when attempting to apply Theorem~\ref{thm:MainResult} to demonstrate the existence of localized solutions to a lattice dynamical system. In the context of (\ref{LDS}), extensive work has shown that steady-state front and back solutions connecting $U_n = 0$ and $U_n = 1$ asymptotically exist for all $d > 0$ in a symmetric parameter region centred about $\mu = 0.5$ and that this region becomes exponentially localized about $\mu = 0.5$ as $d \to \infty$ (see, for example, the review article \cite{Hupkes}). We reiterate that these front and back solutions manifest themselves as heteroclinic orbits of (\ref{LDS_Map}) which asymptotically connect the fixed points $(0,0)$ and $(1,1)$.

Aside from recalling previous studies, we may explicitly confirm Hypothesis~\ref{hyp:Heteroclinic} in the parameter region $0 < d \ll 1$ by perturbing off of the singular parameter value $d = 0$. Indeed, notice that setting $d = 0$ in (\ref{LDS_Steady}) completely decouples elements along the lattice and therefore we may define the singular back solutions $\bar{U}(\mu) = \{\bar{u}_n(\mu)\}_{n\in\mathbb{Z}}$ with
 \begin{equation} \label{HetSol1}
	\bar{U}_n(\mu) = \left\{
     		\begin{array}{cl} 0, & n \leq 0 \\ 
		1, & n > 0
		\end{array}
   	\right.
\end{equation}
and $\bar{V}(\mu) = \{\bar{V}_n(\mu)\}_{n\in\mathbb{Z}}$ with
 \begin{equation} \label{HetSol2}
	\bar{V}_n(\mu) = \left\{
     		\begin{array}{cl} 0, & n < 0 \\ 
		\mu, & n= 0 \\
		1, & n > 0
		\end{array}
   	\right..
\end{equation}
which are solutions of (\ref{LDS_Steady}) when $d = 0$. We note that by construction we have 
\begin{equation}\label{SingularConnections}
	\lim_{\mu\to0^+} \|\bar{U}(\mu) - \bar{V}(\mu)\|_\infty = 0, \quad \lim_{\mu\to0^+} \|S^{-1}\bar{U}(\mu) - \bar{V}(\mu)\|_\infty = 0,
\end{equation}
where $S:\ell^\infty \to \ell^\infty$ is the shift operator defined in (\ref{Shift}). Let us define 
\begin{equation}\label{Gamma0}
	\Gamma_0 := \bigcup_{p \in \mathbb{Z}}\bigcup_{\mu \in [0,1]} \{(S^p\bar{U}(\mu),\mu),(S^p\bar{V}(\mu),\mu)\} \subset \ell^\infty \times [0,1]
\end{equation}
and note that based upon (\ref{SingularConnections}) we have that $\Gamma_0$ is a connected curve. This leads to the following proposition which confirms Hypothesis~\ref{hyp:Heteroclinic} for the system (\ref{LDS}).

\begin{prop}\label{prop:Heteroclinic} %Proposition: Heteroclinic confirmation
	There exists $d_* > 0$ such that for all $0 < d < d_*$ there exists a smooth curve $\Gamma(d) \subset \ell^\infty \times (0,1)$ satisfying the following:
	\begin{enumerate}
		\item For each fixed $d \in (0,d_*)$ every element $(U,\mu) \in \Gamma(d)$ satisfies (\ref{LDS_Steady}) for the given value of $d > 0$ and $U = \{U_n\}_{n\in\mathbb{Z}}$ is such that $U_n \to 0$ as $n \to -\infty$ and $U_n \to 1$ as $n \to \infty$. 
		\item $\pi(\Gamma(d))$ is a closed loop for all $d \in (0,d_*)$.
		\item $\Gamma(d) \to \Gamma_0$ uniformly in the $\ell^\infty\times\mathbb{R}$ norm as $d \to 0^+$.	
	\end{enumerate}
\end{prop} %End of proposition

\begin{proof}
	This proof is carried out the same way as \cite[Proposition~5.1]{Bramburger} and therefore we only outline the steps needed to complete the proof. First, we note that since the roots $\{0,\mu,1\}$ of the nonlinearity $u(u - \mu)(1-u)$ are non-degenerate for all $\mu \in (0,1)$, it follows that the solutions $\bar{U}(\mu)$ and $\bar{V}(\mu)$ can be continued regularly in $0 < d \ll 1$ in any compact subinterval of the interval $(0,1)$ via the implicit function theorem. Therefore, we need only understand how the solutions $\bar{U}(\mu)$ and $\bar{V}(\mu)$ continue near the bifurcation points $\mu = 0,1$. This process can be undertaken using Lyapunov-Schmidt reduction and blow-up techniques in neighbourhoods of $\mu = 0$ and $\mu = 1$ to show that the transcritical bifurcations present when $d = 0$ degenerate into saddle-node bifurcations occurring at values of $\mu$ in the interior of the parameter interval $[0,1]$ when $0 < d \ll 1$. This then gives a smooth curve connecting the continuations of $\bar{U}(\mu)$, $\bar{V}(\mu)$, and $S^{-1}\bar{U}(\mu)$ for $0 < d \ll 1$. We then exploit the equivariance of (\ref{LDS_Steady}) with respect to the shift operator $S$ to obtain the full unbounded curve $\Gamma(d)$. The shift equivariance also gives that $\pi(\Gamma(d))$ is a closed loop.   
\end{proof} %End of proof

\begin{figure} %Figure: Asymmetric bifurcations from 2Pulse isola
\centering
\includegraphics[width=0.9\textwidth]{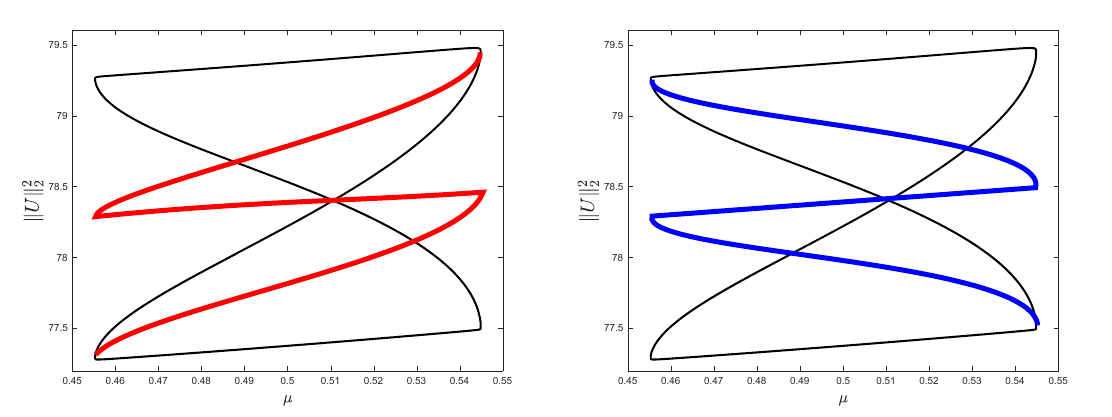}
\caption{Branches of asymmetric $2$-pulses which bifurcate from the curve of symmetric $2$-pulse presented in Figure~\ref{fig:2Pulse}. Both curves originate and terminate at pitchfork bifurcations which take place exponentially close to the saddle-node bifurcations of opposite curvature along the symmetric $2$-pulse curve.}
\label{fig:2Pulse_Pitchfork}
\end{figure} 

{\bf Isolas of Multipulses.} Following the discussion proceeding Hypothesis~\ref{hyp:Heteroclinic}, we find that since for each fixed $0 < d \ll 1$ the set $\Gamma(d)$ is not a closed loop, single-pulse solutions of (\ref{LDS}) snake, as demonstrated in Figure~\ref{fig:Snaking}. Our work in this manuscript shows that regardless of the bifurcation structure of single-pulse solutions, all multi-pulse solutions of (\ref{LDS}) lie along closed curves in $\ell^\infty \times [0,1]$. This is illustrated in Figure~\ref{fig:2Pulse} where an isola of symmetric $2$-pulse solutions of (\ref{LDS}) is provided. Moreover, since the bifurcation curve in Figure~\ref{fig:2Pulse} represents symmetric solutions, Theorem~\ref{thm:MainResult}(2) dictates that near each of the four saddle-node bifurcations we expect a symmetry-breaking pitchfork bifurcation to occur. The continued symmetry-breaking curves are presented in Figure~\ref{fig:2Pulse_Pitchfork} where we can see that they form distinctive `zig-zag' patterns connecting saddle-nodes at opposite extremities of the symmetric isola. This again is consistent with the results of Theorem~\ref{thm:MainResult} since we see that the pitchfork bifurcations which mark the endpoints of the asymmetric curves take place at saddle-node bifurcations of opposite curvature along the branch. One minor shortcoming of Theorem~\ref{thm:MainResult} is that it fails to provide explicit information as to where exactly a bifurcating asymmetric branch will terminate, but we direct the reader to the work of \cite[Section~5.2]{Beck} where the authors provide methods of visually and analytically continuing such branches.      

\begin{figure} %Figure: Asymmetric 2-Pulses
\centering
\includegraphics[width=\textwidth]{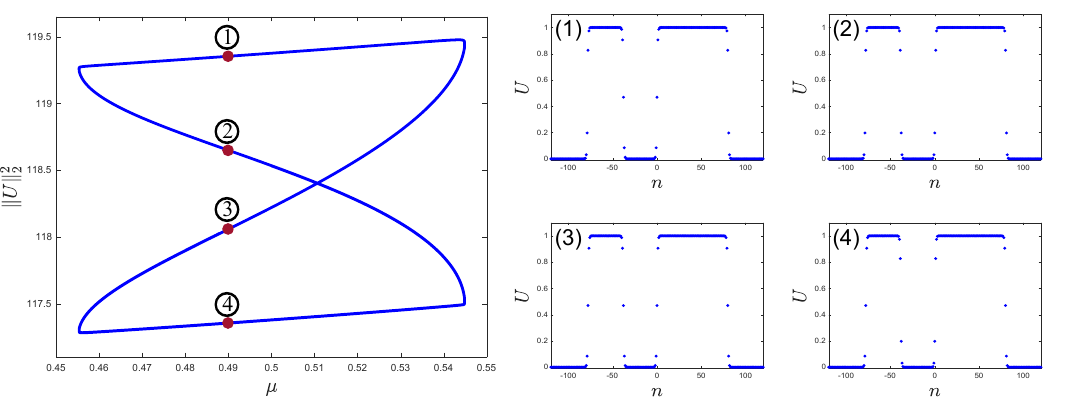}
\caption{An isola of asymmetric $2$-pulse solutions to (\ref{LDS}) with $d = 0.1$. In the context of Theorem~\ref{thm:MainResult} these solutions represent a homoclinic orbit of (\ref{LDS_Map}) with $N_1 < N_3$. The asymmetry of the $2$-pulses give that no pitchfork bifurcations take place anywhere along the closed bifurcation curve. Sample profiles are provided for $\mu = 0.49$ along the bifurcation curve.}
\label{fig:2Pulse_Asym}
\end{figure} 

Theorem~\ref{thm:MainResult} of course goes far beyond symmetric $2$-pulses to provide both the existence and bifurcation structure of a denumerable number of multi-pulse solutions of (\ref{LDS}). Figures~\ref{fig:2Pulse_Asym} and \ref{fig:5Pulse} provide further numerical confirmation of the results in this manuscript. Figure~\ref{fig:2Pulse_Asym} presents an isola of asymmetric $2$-pulses which have $N_1 < N_3$, using the notation of Theorem~\ref{thm:MainResult}. We note that since these $2$-pulses are asymmetric, it is guaranteed that their bifurcation curve contains no symmetry-breaking bifurcation branches and hence the curve given in Figure~\ref{fig:2Pulse_Asym} is the entire connected bifurcation curve of the associated asymmetric $2$-pulses. To move beyond $2$-pulses, Figure~\ref{fig:5Pulse} presents two more bifurcation curves: one for asymmetric $4$-pulses and one for symmetric $5$-pulses. Interestingly, these bifurcation curves bear little resemblance to the hourglass shape observed in our numerics for $2$-pulses, but is again nonetheless an isola. An interesting avenue for future exploration would be to determine the mechanism that dictates the shape of the resulting isola.    

\begin{figure} %Figure: Asymmetric 4-Pulses and Symmetric 5-Pulses
\centering
\includegraphics[width=\textwidth]{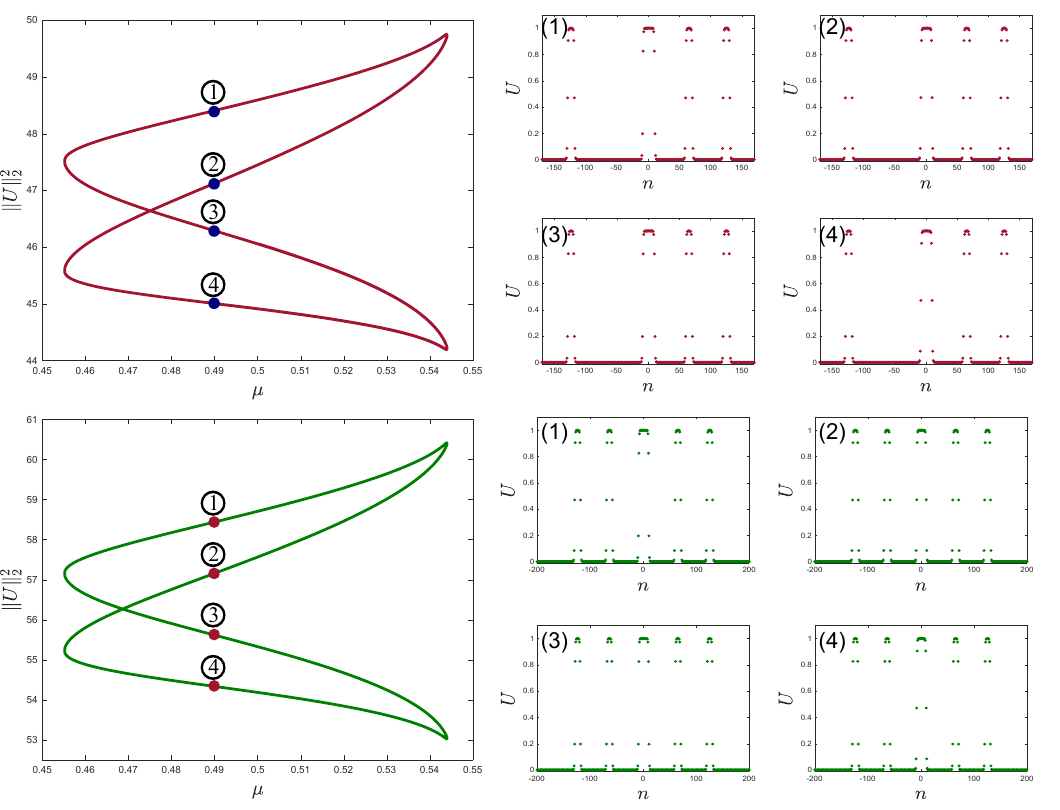}
\caption{An isola of asymmetric $4$-pulse (red) and symmetric $5$-pulse (green) solutions to (\ref{LDS}) with $d = 0.1$. Sample profiles are provided at $\mu = 0.49$ along the bifurcation curve.}
\label{fig:5Pulse}
\end{figure}

%\subsection{Stability}\label{subsec:Stability} %Subsection: Stability ---------------------------------------------------------------------------------------------------------------------------------

\begin{figure} %Figure: Stability
\centering
\includegraphics[width=\textwidth]{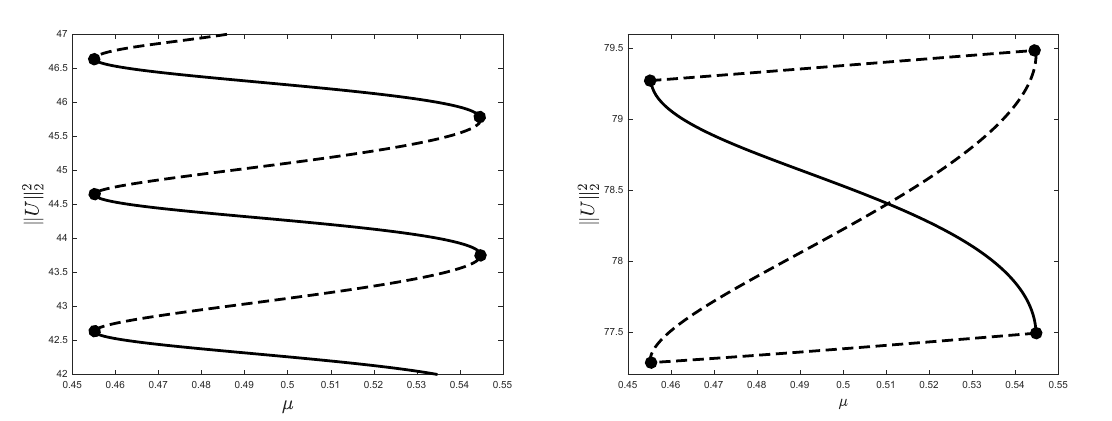}
\caption{Numerical calculation of spectral stability along the bifurcation curves of single and $2$-pulse solutions of (\ref{LDS}). Solid lines represent spectrally stable solutions, dashed lines are unstable solutions, and large dots represent the saddle-node bifurcations where stability changes along the branch. On the left we present the isolated bifurcation curve of on-site single-pulses from Figure~\ref{fig:Snaking} and on the right we present the bifurcation curve of the symmetric $2$-pulses from Figure~\ref{fig:2Pulse}.}
\label{fig:Stability}
\end{figure} 

{\bf Stability.} Let us now briefly discuss the expected stability of the multi-pulse solutions described in this manuscript. First, recent results in the continuous spatial setting have shown that the stability of single-pulse solutions can be derived from the front and back solutions used to demonstrate their existence \cite{Makrides}. This work uses gluing arguments to show that single-pulse solutions are exponentially close to a front and back solution glued together. This exponential closeness endowed by the method of gluing then is used to show that isolated eigenvalues associated to the linearization about a front or back solution used to create the localized solution lead to isolated eigenvalues of the associated localized solution. An immediate consequence of this fact is that if a localized solution is created by gluing an unstable front solution to a back solution, or vice-versa, then it must also be unstable as well. We leave the full analytical extension of the results of \cite{Makrides} to a future exposition and only briefly comment on their expected implications here.

Linearizing (\ref{LDS}) about the solutions $\bar{U}(\mu)$ and $\bar{V}(\mu)$ at $d = 0$ trivially gives that for all $\mu \in (0,1)$ the solutions $\bar{U}(\mu)$ have spectrum entirely contained in the negative real numbers while the solutions $\bar{V}(\mu)$ have exactly one positive real eigenvalue and all others belonging in the negative reals. Hence, traversing the curve $\Gamma_0$ from (\ref{Gamma0}) gives a curve of steady-state solutions of (\ref{LDS}) at $d = 0$ with a single real eigenvalue drifting back and forth across $0$. The results of Proposition~\ref{prop:Heteroclinic} and the boundedness of the coupling function in (\ref{LDS}) can be used to infer that the same happens as one traverses the curves $\Gamma(d)$ for all $0 < d \ll 1$. Hence, curve $\Gamma(d)$ is composed of spectrally stable and unstable back solutions of (\ref{LDS}) for $0 < d \ll 1$ which meet at saddle-node bifurcations near $\mu = 0$ and $\mu = 1$. Of course, reflecting elements of $\Gamma(d)$ over the index $n = 0$ gives the exact same behaviour for steady-state front solutions of (\ref{LDS}). This symmetry is exactly what gives the reversibility of (\ref{LDS_Map}) and is exploited in Section~\ref{sec:Proofs} to prove Theorem~\ref{thm:MainResult}.     

The process of gluing a back solution and its associated front solution obtained by reflection over $n = 0$ is exactly how symmetric single-pulse solutions of (\ref{LDS}) can be constructed. Hence, using the work of \cite{Makrides} as a guide, one expects that ascending the snaking curve of symmetric single pules provided in Figure~\ref{fig:Snaking} results in alternating branches of spectrally stable and unstable solutions which collide at the left and right saddle-node bifurcations. This is confirmed numerically for on-site solutions in Figure~\ref{fig:Stability}. Furthermore, the gluing process results in the unstable single-pulses having exactly two unstable eigenvalues: one from the front solution and one from the back solution. As one ascends the bifurcation curve both eigenvalues cross zero, resulting in two distinct steady-state bifurcations: the saddle-node along the symmetric branch and the pitchfork to the pair of asymmetric single-pulses. The asymmetric single-pulses can be viewed as gluing a back and front solution which do not reflect into each other (up to discrete translation along the lattice), which in this scenario means one of the front or back must be unstable while the other is spectrally stable. Hence, asymmetric single-pulse solutions of (\ref{LDS}) should be expected to be unstable and numerics appear to confirm this. 

It was shown in \cite{2Pulse} that $2$-pulse solutions in the continuous spatial setting can be viewed as gluing together two single-pulse solutions and although the work \cite{Makrides} only covers stability of single-pulse solutions, its methods should be applicable to multi-pulses as well. Furthermore, the process of determining the stability of multi-pulses from associated single-pulses is handled in \cite{Sandstede} and hence can act as a guide for our discussion here. In the present context of (\ref{LDS}), the creation of a $2$-pulse from single-pulses can be observed in Figure~\ref{fig:2Pulse} where it appears that the symmetric $2$-pulse solutions are composed of two single-pulses with one reflected and glued to the other. Hence, stability of the associated $2$-pulses should follow from the stability of single-pulse solutions of (\ref{LDS}). More precisely, it should hold that the top and bottom pieces of the bifurcation curve in Figure~\ref{fig:2Pulse} are unstable with two positive real eigenvalues since they are formed from asymmetric single-pulses which have a single positive real eigenvalue each. One of the branches that connect the top and the bottom of the bifurcation curve should be stable since it is composed of $2$-pulses which are formed from stable single-pulses, while the other should be composed of unstable $2$-pulses for similar reasons. This is confirmed numerically in Figure~\ref{fig:Stability}. Similarly, tracking the number of eigenvalues that cross the imaginary axis as one traverses each component of the bifurcation curve leads to an intuitive understanding of why the pitchfork bifurcations described in Theorem~\ref{thm:MainResult} happen so close to the saddle-node bifurcations. For similar reasons to the single-pulse case, we expect that all bifurcating asymmetric $2$-pulses are unstable.     

With the intuition we have built up in the previous paragraphs, it now becomes a straightforward mental exercise of determining which multi-pulse solutions to (\ref{LDS}) are expected to be stable and which are expected to be unstable. It is worth mentioning that one may be led to conjecture that all asymmetric multi-pulses should be expected to be unstable, but this should not be the case. That is, one may have an asymmetric $2$-pulse solution which is formed by gluing a spectrally stable single-pulse to another spectrally stable single-pulse with a significantly longer plateau. This would be the process used to create the solutions in Figure~\ref{fig:2Pulse_Asym} and numerical investigations reveal that the stability along this bifurcation curve should be the same as that which is presented on the right of Figure~\ref{fig:Stability} for symmetric $2$-pulses.   

This discussion of stability is entirely formal and was partially backed up by the numerics presented in Figure~\ref{fig:Stability}. A full analytical treatment of the stability of localized patterns in lattice dynamical systems will be left to a subsequent study. As a final note, spectral stability is all that is required to conclude local asymptotic stability of a localized solution to (\ref{LDS}). The reason for this is that the linearization about a localized solution to (\ref{LDS}) results in a bounded operator, and hence standard theory gives that the semi-group generated by such a bounded linear operator with spectrum lying entirely to the left of the imaginary axis decays exponentially in time. This uniform exponential decay of the semi-group can then be extended to small perturbations from the associated localized steady-state of the lattice dynamical system via standard arguments. Hence, determining spectral stability of localized solutions of (\ref{LDS}) guarantees local asymptotic stability of the solution as well.

\section{Proofs}\label{sec:Proofs}%Section: Proofs ---------------------------------------------------------------------------------------------------------------------------------

Throughout this section we provide the proof of Theorem~\ref{thm:MainResult} by breaking it down into a series of smaller results. In \S~\ref{sec:u_*} we review some results form \cite{Bramburger} that transform the dynamics of (\ref{Map}) in a neighbourhood of the fixed point $u_*$ to better describe the local dynamics. We then extend these coordinate transformations to a neighbourhood of the trivial fixed point in \S~\ref{sec:0} via the same methods as the previous subsection. Due to the similar nature of the results between \S~\ref{sec:u_*} and \S~\ref{sec:0}, throughout this section constants with a $*$ subscript correspond to results in the neighbourhood of the fixed point $u = u_*$ and constants with a $0$ subscript correspond to results in the neighbourhood of the fixed point $u = 0$. The existence and bifurcation structure of symmetric  $2$-pulses are left to \S~\ref{subsec:Sym2}, whereas the asymmetric $2$-pulses are handled in \S~\ref{subsec:Asym2}. Then \S~\ref{subsec:kPulse} extends the results for $2$-pulses to $k$-pulses for arbitrary $k \geq 3$.

\subsection{Local Coordinates About $u_*$}\label{sec:u_*} %Section: u_* ----------------------------------------------------------------------------------------------------------------------------------------------------------

In this section we characterize the dynamics near the fixed point $u_*$ by recalling the work of \cite{Bramburger}. We begin by noting that Hypotheses~\ref{hyp:Reverser}-\ref{hyp:FixedPts} imply that the eigenvalues of $F_u(u_*,\mu)$ are of the form $0 < \lambda_*(\mu)^{-1} < 1 < \lambda_*(\mu)$ for some smooth function $\lambda_*(\mu)$. The following result was proved in \cite{Bramburger} and uses normal hyperbolicity of $u_*$ for all $\mu \in J$ to provide a near-identity change of coordinates to characterize the local dynamics in a simpler way.  

\begin{lem}[\cite{Bramburger}]\label{lem:Shilnikov*} %Lemma: Shilnikov* Variables
	Assume Hypotheses~\ref{hyp:Reverser} and \ref{hyp:FixedPts} are met. Then, there exists $\delta_* > 0$, a smooth change of coordinates mapping $u$ to $v = (v^s,v^u)$ near the fixed point $u = u_*$, and smooth functions $f^s_i,f^u_i:\mathcal{I}_*\times\mathcal{I}_*\times J\to \mathbb{R}$, $i = 1,2$, so that (\ref{Map}) is of the form 
	\begin{equation}\label{Shil*}
		\begin{split}
			v^s_{n+1} &= [\lambda_*(\mu)^{-1} + f_1^s(v^s,v^u,\mu)v^s_n + f_2^s(v^s,v^u,\mu)v^u_n]v^s_n, \\
			v^u_{n+1} &= [\lambda_*(\mu) + f_1^u(v^s,v^u,\mu)v^s_n + f_2^u(v^s,v^u,\mu)v^u_n]v^u_n, \\
		\end{split}
	\end{equation}	
	for all $\mu \in J$, where $v^s_n,v^u_n \in \mathcal{I}_* := [-\delta_*,\delta_*]$, and the reverser $\mathcal{R}$ acts by 
	\begin{equation}\label{Shil*Reverser}
		\mathcal{R}(v^s,v^u) = (v^u,v^s).
	\end{equation}
\end{lem}

We can further characterize solutions of (\ref{Shil*}) using the following lemma. 

\begin{lem}[\cite{Bramburger}]\label{lem:Shil*Sol} %Lemma: Shilnikov* Solution
	There exists constants $\eta_* \in (0,1)$ and $M_* > 0$ such that the following is true: for each $N > 0$, $a^u,a^s\in\mathcal{I}_*$, and $\mu \in J$ there exists a unique solution near the origin to (\ref{Shil*}), written $v_n = (v^s_n,v^u_n) \in \mathcal{I}_*\times\mathcal{I}_*$ with $n \in \{0,\dots,N\}$, such that
	\[
		v^s_0 = a^s, \quad v^u_N = a^u.
	\]
	Furthermore, this solution satisfies
	\begin{equation}\label{Shil*Bnds}
		|v^s_n| \leq M_*\eta_*^n, \quad |v^u_n| \leq M_*\eta_*^{N-n},
	\end{equation}
	for all $n \in \{0,\dots,N\}$, $v_n = v_n(a^s,a^u,\mu)$ depends smoothly on $(a^s,a^u,\mu)$, and the bounds (\ref{Shil*Bnds}) also hold for the derivatives of $v$ with respect to $(a^s,a^u,\mu)$. Moreover, 
	\[
		\mathcal{R}(v^s_n,v^u_n) = (v^u_{N-n},v^s_{N-n}),
	\]
	for all $n \in \{0,\dots,N\}$. In particular, the solution $v$ is symmetric if, and only if, $a^s = a^u$.
\end{lem}

Now, notice that the positivity of the eigenvalues of $DF_u(u_*,\mu)$ for all $\mu \in J$ assumed in Hypothesis~\ref{hyp:FixedPts} implies that the stable and unstable manifolds of the fixed point $u_*$ are orientation preserving. In \cite{Bramburger} this fact was used to construct an interval $K_0 := [\delta_L,\delta_R]$, with $\delta_L,\delta_R \in (0,\delta)$, such that for all $\mu \in J$ we have the following:
\begin{enumerate}
	\item The backward iteration (\ref{InverseMap}), $F^{-1}$, maps the point $(v^s,v^u) = (\delta_L,0)$ into the interval $(\delta_L,\delta_R)\times\{v^u = 0\}$.
	\item The forward iteration (\ref{Map}), $F$, maps the point $(v^s,v^u) = (\delta_R,0)$ into the interval $(\delta_L,\delta_R)\times\{v^u = 0\}$.
	\item The backward iteration (\ref{InverseMap}), $F^{-1}$, maps the point $(v^s,v^u) = (\delta_R,0)$ out of the set $\mathcal{I}\times\mathcal{I}$.  
\end{enumerate}
Furthermore, consider a closed interval $K_1$ with the property that $K_0 \subset K_1 \Subset (0,\delta)$. This allows for the definition of the segment
\[
	\Sigma_\mathrm{in} := K_1 \times \mathcal{I}.
\]
Applying the reverser $\mathcal{R}$ and using the action (\ref{Shil*Reverser}) allows one to further define 
\[
	\Sigma_\mathrm{out} := \mathcal{I} \times K_1.
\]

It is a straightforward task to show that the choices of $\delta_L,\delta_R$ are sufficient to show that for each fixed $\mu \in J$ we have 
\[
	W^u(0,\mu) \cap \{(v^s,0)\in\mathcal{I}\times\mathcal{I}:\ v^s \in (0,\delta)\} \neq \emptyset
\]
if, and only if, 
\[
	W^u(0,\mu) \cap \{(v^s,0)\in\mathcal{I}\times\mathcal{I}:\ v^s \in K_0\} \neq \emptyset.
\]
We illustrate these facts in Figure~\ref{fig:Local_u*}, and tracing these intersections for varying $\mu$ allows one to define the local component of $\Gamma$ to $K_0$, denoted $\Gamma_\mathrm{loc}$, given by
\[
	\Gamma_\mathrm{loc} := \bigcup_{\mu \in J} (W^u(0,\mu) \cap \{(v^s,0):v^s\in K_0\}) \subset K_0 \times \{v^u = 0\}\times \mathring{J}.	
\]
We note that $\Gamma_\mathrm{loc}$ is closed and nonempty and represents the piece of the heteroclinic orbits connecting $0$ to $u_*$ that lies in $K_0\times\{v^u=0\}$ for each $\mu \in J$. Our above discussion implies that for any $\mu \in J$ for which a heteroclinic orbit connecting $0$ to $u_*$ exists, then $\Gamma_\mathrm{loc}$ is nonempty for this value of $\mu$. This leads to the following results, proven in \cite{Bramburger}.

\begin{figure} %Figure: Local intersection of manifolds about u_*
\centering
\includegraphics[width=0.3\textwidth]{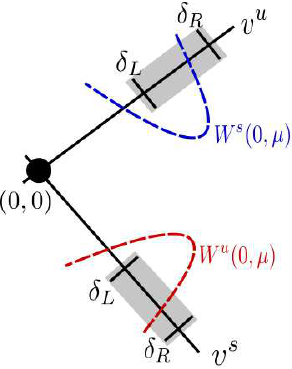}
\caption{The choices of $\delta_L,\delta_R$ guarantee that $W^u(0,\mu)$, the unstable manifold of $0$, must intersect $K_0\times\{0\} \subset \mathcal{I}_*\times\mathcal{I}_*$ for any value of $\mu$ for which $W^u(0,\mu) \cap W^s(u_*,\mu) \neq \emptyset$. Reversibility further implies that $W^s(0,\mu)$, the stable manifold of $0$, must intersect $\{0\}\times K_0\subset \mathcal{I}_*\times\mathcal{I}_*$ for any value of $\mu$ for which $W^s(0,\mu) \cap W^u(u_*,\mu) \neq \emptyset$. The shaded boxes represent $K_1 \times \{|v^u| < \varepsilon\}$ and $\{|v^s| < \varepsilon\} \times K_1$ from Lemma~\ref{lem:G*Fn}.}
\label{fig:Local_u*}
\end{figure} 

\begin{lem}[\cite{Bramburger}]\label{lem:G*Fn} %Lemma: G* Function
	There exists an $\varepsilon_* > 0$ and a function $G_*:K_1\times\mathcal{I}\times J\to\mathbb{R}$ such that $G_*(v^s,v^u,\mu) = 0$ if, and only if, $(v^s,v^u,\mu) \in (W^u(0,\mu)\times\{\mu\})$ with $|v^u| < \varepsilon_*$. Furthermore, 
	\[
		\nabla_{(v^s,\mu)} G_*(v^s,0,\mu) \neq 0
	\]
	for all $(v^s,0,\mu) \in \Gamma_\mathrm{loc}$ and $\partial_{v^s}G_*(v^s,0,\mu) \neq 0$ if, and only if, $(v^s,0,\mu) \in \Gamma_\mathrm{loc}$ represents a point of transverse intersection between the manifolds $W^u(0,\mu)$ and $W^s(u_*,\mu)$.
\end{lem}

\subsection{Local Coordinates About $0$}\label{sec:0} %Section: 0 ----------------------------------------------------------------------------------------------------------------------------------------------------------

In this section we derive similar results to the previous section to show that the dynamics local to the fixed point $u = 0$ can be handled in a very similar way to those of the fixed point $u = u_*$. Here again Hypothesis~\ref{hyp:FixedPts} gives that the eigenvalues of $DF_u(0,\mu)$ are of the form $0 < \lambda_0(\mu)^{-1} < 1 < \lambda_0(\mu)$, for some smooth function $\lambda_0(\mu)$, for all $\mu \in J$. We begin by providing the analogues of Lemmas~\ref{lem:Shilnikov*} and \ref{lem:Shil*Sol} near $u = 0$. Both results are stated without proof since their proof is identical to their $u = u_*$ analogues.

\begin{lem}\label{lem:Shilnikov0} %Lemma: Shilnikov0 Variables
	Assume Hypotheses~\ref{hyp:Reverser} and \ref{hyp:FixedPts} are met. Then, there exists $\delta_0 > 0$, a smooth change of coordinates mapping $u$ to $v = (w^s,w^u)$ near the fixed point $u = 0$, and smooth functions $g^s_i,g^u_i:\mathcal{I}_0\times\mathcal{I}_0\times J\to \mathbb{R}$, $i = 1,2$, so that (\ref{Map}) is of the form 
	\begin{equation}\label{Shil0}
		\begin{split}
			w^s_{n+1} &= [\lambda_0(\mu)^{-1} + g_1^s(w^s,w^u,\mu)w^s_n + g_2^s(w^s,w^u,\mu)w^u_n]w^s_n, \\
			w^u_{n+1} &= [\lambda_0(\mu) + g_1^u(w^s,w^u,\mu)w^s_n + g_2^u(w^s,w^u,\mu)w^u_n]w^u_n, \\
		\end{split}
	\end{equation}	
	for all $\mu \in J$, where $w^s_n,w^u_n \in \mathcal{I}_0 := [-\delta_0,\delta_0]$, and the reverser $\mathcal{R}$ acts by 
	\begin{equation}\label{Shil0Reverser}
		\mathcal{R}(w^s,w^u) = (w^u,w^s).
	\end{equation}
\end{lem}

\begin{lem}\label{lem:Shil0Sol} %Lemma: Shilnikov0 Solution
	There exists constants $\eta_0 \in (0,1)$ and $M_0 > 0$ such that the following is true: for each $N > 0$, $b^u,b^s\in\mathcal{I}_0$, and $\mu \in J$ there exists a unique solution near the origin to (\ref{Shil0}), written $w_n = (w^s_n,w^u_n) \in \mathcal{I}_0\times\mathcal{I}_0$ with $n \in \{0,\dots,N\}$, such that
	\[
		w^s_0 = b^s, \quad w^u_N = b^u.
	\]
	Furthermore, this solution satisfies
	\begin{equation}\label{Shil0Bnds}
		|w^s_n| \leq M_0\eta_0^n, \quad |w^u_n| \leq M_0\eta_0^{N-n},
	\end{equation}
	for all $n \in \{0,\dots,N\}$, $w_n = w_n(b^s,b^u,\mu)$ depends smoothly on $(b^s,b^u,\mu)$, and the bounds (\ref{Shil*Bnds}) also hold for the derivatives of $w$ with respect to $(b^s,b^u,\mu)$. Moreover, 
	\begin{equation}\label{Shil0SolReverser}
		\mathcal{R}(w^s_n,w^u_n) = (w^u_{N-n},w^s_{N-n}),
	\end{equation}
	for all $n \in \{0,\dots,N\}$. In particular, the solution $w$ is symmetric if, and only if, $b^s = b^u$.
\end{lem}

\begin{figure} %Figure: Pull-back mapping
\centering
\includegraphics[width=0.6\textwidth]{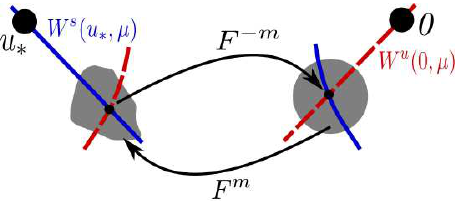}
\caption{A visual depiction of the proof of Lemma~\ref{lem:mdefn} for a fixed $\mu \in J$. A point of intersection along $W^s(u_*,\mu) \cap W^u(0,\mu)$ in a neighbourhood of $u_*$ can be mapped back to a neighbourhood of $0$ by $F^{-m}$, for a fixed integer $m$. Similarly, the shaded disk on the right representing a neighbourhood of a point of intersection of $W^s(u_*,\mu) \cap W^u(0,\mu)$ near $0$ can be mapped forward to a neighbourhood near $u_*$ by $F^m$.}
\label{fig:PullBack}
\end{figure} 

We now provide the following lemma relating solutions near the fixed point $0$ to fixed points $u_*$. The reader is referred to Figure~\ref{fig:PullBack} for a visual guide to the proof of the results.

\begin{lem}\label{lem:mdefn}
	There exists a fixed positive integer $m \geq 1$ such that $F^{-m}(\Gamma_\mathrm{loc}) \subset \{0\}\times\mathring{\mathcal{I}}_0\times\mathring{J}$. Furthermore, there exists $\varepsilon_0 > 0$ such that the open set
	\[
		U_0 := \{(w^s,w^u,\mu)\in\mathcal{I}_0\times\mathcal{I}_0\times J: \mathrm{dist}((w^s,w^u,\mu),F^{-m}(\Gamma_\mathrm{loc}))<\varepsilon_0\}
	\]  
	is such that $F^m(U_0) \subset \Sigma_\mathrm{in}\times\mathring{J}$.
\end{lem}

\begin{proof}
	We begin by noting that since $\Gamma_\mathrm{loc}$ contains elements lying on the intersections of $W^u(0,\mu)$ and $W^s(u_*,\mu)$ for varying values of $\mu$, for each point $(u,\mu) \in \Gamma_\mathrm{loc}$ we have that $F^{-m}(u,\mu) \to 0$ as $m \to \infty$. In particular, for each point $(u,\mu) \in \Gamma_\mathrm{loc}$ there exists some $m_0 \geq 1$ such that  $F^{-m}(u,\mu) \in \{0\}\times\mathcal{I}_0$ for all $m \geq m_0$. Since $\Gamma_\mathrm{loc}$ is compact, it follows that there exists a minimal value of $m \geq 1$ such that $F^{-m}(\Gamma_\mathrm{loc}) \subset \{0\}\times\mathring{\mathcal{I}}_0\times\mathring{J}$. The final claim of the lemma simply follows from the fact that $\Gamma_\mathrm{loc}$ lies in the interior of $ \Sigma_\mathrm{in}\times\mathring{J}$ and the fact that $F$ is a diffeomorphism, i.e. it has a continuous inverse.    
\end{proof}

Note that reversibility of (\ref{Map}) implies that $F^{-m}(\mathcal{R}U_0) \subset \Sigma_\mathrm{out}\times\mathring{J}$. This fact will will be used throughout the proofs in the following subsections.

\subsection{Symmetric 2-Pulses}\label{subsec:Sym2} %Subsection: Symmetric 2-Pulses ---------------------------------------------------------------------------------------------------------------------------

We begin by proving the existence of symmetric  2-pulses to $u = 0$ to best illustrate the methods. Here, by definition, a symmetric  2-pulse $u = \{u_n\}_{n\in\mathbb{Z}}$ satisfies
\begin{subequations}\label{2Pulse}
	\begin{equation}\label{2Pulse1}
		u_n \in \mathcal{I}_0\times\mathcal{I}_0\ {\rm for}\ n\in\{0,\dots,N_2\}	\ \mathrm{with}\ u_n = \mathcal{R}u_{N_2 - n}
	\end{equation}
	\begin{equation}\label{2Pulse2}
		u_{N_2} \in W^s(0,\mu) 	
	\end{equation}
	\begin{equation}\label{2Pulse3}
		u_{N_2+m} \in \Sigma_\mathrm{in}\cap W^u(0,\mu) 	
	\end{equation}
	\begin{equation}\label{2Pulse4}
		u_n \in \mathcal{I}_*\times\mathcal{I}_*\ {\rm for}\ n\in\{N_2+m,\dots,N_2 + N_3+m\}	
	\end{equation}
	\begin{equation}\label{2Pulse5}
		u_{N_2+N_3+m} \in \Sigma_\mathrm{out}\cap W^s(0,\mu) 	
	\end{equation}
\end{subequations}  
for sufficiently large $N_2,N_3 \gg 1$, and $m \geq 0$ is the integer defined in Lemma~\ref{lem:mdefn}. We note that reversibility of $F$ and (\ref{2Pulse1}) imply that 
\[
	u_{\frac{N_2}{2}+n} = \mathcal{R}u_{\frac{N_2}{2}-n}
\] 
or 
\[
	u_{\frac{N_2}{2}+1+n} = \mathcal{R}u_{\frac{N_2}{2}-n}
\] 
for all $n \in \mathbb{Z}$, depending on the parity of $N_2$. Therefore, a homoclinic orbit satisfying (\ref{2Pulse1}) necessarily is symmetric and hence, following the iterates backward from $n = 0$ results in another sequence of iterations of length $N_3$ for which the orbit remains in the neighbourhood $\mathcal{I}_*\times\mathcal{I}_*$ of the fixed point $u = u_*$. These iterations are separated by $N_2$ iterations in the neighbourhood $\mathcal{I}_0\times\mathcal{I}_0$ of the fixed point $u=0$, and therefore using the terminology of Theorem~\ref{thm:MainResult} such a symmetric  orbit corresponds to the integers $N_1,N_2,N_3$ with $N_2,N_3$ as above and $N_1 = N_3$. 

We present the following result which will be used throughout this subsection and those which follow. It is stated without proof since it is proven via similar methods to many root-finding results.   

\begin{thm}\label{thm:Roots} %Theorem: Kantorovich-style Roots
	Let $H:\mathbb{R}^d \to \mathbb{R}^d$ be a smooth function and assume there exist an invertible matrix $A \in \mathbb{R}^{d\times d}$, $x_0 \in \mathbb{R}^d$, $0 < \kappa < 1$, and $\rho > 0$ such that
	\begin{enumerate}
		\item $\|1 - A^{-1}DH(x)\| \leq \kappa$ for all $x \in B_\rho(x_0)$,
		\item $\|A^{-1}H(x_0)\|\leq (1-\kappa)\rho$.
	\end{enumerate}
	Then $H$ has a unique root $x_*$ in $B_\rho(x_0)$, and $|x_* - x_0| \leq \frac{1}{1 - \kappa}\|A^{-1}H(x_0)\|$.
\end{thm} 

\begin{lem}\label{lem:Sym2Pulse}%Lemma: Existence of symmetric 2-pulses
	Assume Hypothesis~\ref{hyp:Reverser}-\ref{hyp:Heteroclinic} and that at $\mu = \bar{\mu}$ the manifolds $W^u(0,\bar{\mu})$ and $W^s(u_*,\bar{\mu})$ intersect transversely at the point $(\bar{v}^s,0,\bar{\mu}) \in \Gamma_\mathrm{loc}$. Then, there exists $M_2^\mathrm{s} > 0$ such that for all $N_2,N_3 > M_2^\mathrm{s}$ the map (\ref{Map}) evaluated at $\mu=\bar{\mu}$ has a symmetric  2-pulse solution. Furthermore, the solution is on-site if $N_2$ is odd and off-site if $N_2$ is even.
\end{lem}

\begin{proof}
	Throughout this proof we will fix $\mu = \bar{\mu}$, and therefore suppress the dependence of solutions on $\mu$. Furthermore, taking $(\bar{v}^s,0,\bar{\mu}) \in \Gamma_\mathrm{loc}$, we use Lemma~\ref{lem:mdefn} to define $\bar{w}^u$ so that 
	\[
		(\bar{w}^u,0) = F^{-m}((\bar{v}^s,0),\bar{\mu}) \in (\mathcal{I}_0\times\mathcal{I}_0)\cap W^u(0,\bar{\mu}) \cap W^s(u_*,\bar{\mu}).
	\]  
	Then, using Lemma~\ref{lem:Shil0Sol} we see that for arbitrary $b^u$ sufficiently small and every integer $N_2 \geq 1$, we have the existence of a reversible solution to (\ref{Shil0}), here denoted as $\{(w^s,w^u)\}_{n=0}^{N_2}\subset\mathcal{I}_0\times\mathcal{I}_0$ satisfying 
	\[
		\begin{split}
			w^s_n &= w^u_{N_2 - n}, \\
			w^s_0 &= w^u_{N_2} = \bar{w}^u + b^u,
		\end{split}
	\]
	where $w^j_n = w^j_n(b^u)$ depend smoothly on $b^u$ in a neighbourhood of $0$ for $j = s,u$ and all $n \in \{0,\dots,N_2\}$. Furthermore, from Lemma~\ref{lem:Shil0Sol} we have that there exists $M_0 > 0$ and $\eta_0 \in (0,1)$ such that 
	\begin{equation}\label{wBnd}
		|w^s_{N_2}(b^u)| \leq M_0\eta_0^{N_2},
	\end{equation}
	uniformly in $b^u$ in a neighbourhood of $0$, and the bound (\ref{wBnd}) holds for all partial derivatives of $w^s_{N_2}(b^u)$ with respect to $b^u$. It now follows from (\ref{Shil0SolReverser}) that when $N_2$ is odd the solution is on-site and when $N_2$ is even the solution is off-site.   

	Similarly, using Lemma~\ref{lem:Shil*Sol} we see that for sufficiently small $a^s,a^u$ and every integer $N_3 \geq 1$, we have the existence of a solution to (\ref{Shil*}), here denoted as $\{(v^s,v^u)\}_{n = 0}^{N_3}\subset\mathcal{I}_*\times\mathcal{I}_*$ satisfying
	\[
		\begin{split}
			v^s_0 &= \bar{v}^s + a^s, \\
			v^u_{N_3} &= \bar{v}^s +a^u, 
		\end{split}
	\] 
	where $v^j_n = v^j_n(a^s,a^u)$ depend smoothly on $(a^s,a^u)$ in a neighbourhood of $(0,0)$ for $j = s,u$ and all $n \in \{0,\dots,N_3\}$. Moreover, Lemma~\ref{lem:Shil*Sol} gives the existence of constants $M_* > 0$ and $\eta_* \in (0,1)$ such that 
	\begin{equation}\label{vBnd}
		|v^u_{0}(a^s,a^u)|, |v^s_{N_3}(a^s,a^u)| \leq M_*\eta_*^{N_3},  
	\end{equation}
	for all sufficiently small $a^s,a^u$ that guarantee $\bar{v}^s + a^s,\bar{v}^s + a^u \in \mathcal{I}_*$, and the bound (\ref{vBnd}) holds for all partial derivatives of $v^u_{0}(a^s,a^u)$ and $v^s_{N_3}(a^s,a^u)$ with respect to $(a^s,a^u)$.
	
	Now, to satisfy the conditions of (\ref{2Pulse}), we begin by taking $N_2,N_3 \geq 1$ sufficiently large so that using (\ref{wBnd}) and (\ref{vBnd}) we can guarantee 
	\[
		(w^s_{N_2}(b^u),\bar{w}^u + b^u,\bar{\mu})\in U_0, \quad |v^u_{0}(a^s,a^u)|, |v^s_{N_3}(a^s,a^u)| < \varepsilon_*, 	
	\]
	for all $(b^u,a^s,a^u)$. Lemma~\ref{lem:G*Fn} and reversibility of (\ref{Map}) then imply that satisfying (\ref{2Pulse5}) becomes equivalent to solving 
	\begin{equation}\label{2PulseMatch1}
		G_*(v^u_{N_3}(a^s,a^u),v^s_{N_3}(a^s,a^u),\bar{\mu}) = G_*(\bar{v}^s + a^u,\mathcal{O}(\eta_*^{N_3}),\bar{\mu}) = 0,
	\end{equation}
	since $\Sigma_\mathrm{out}\cap W^s(0,\mu) = \mathcal{R}(\Sigma_\mathrm{in}\cap W^u(0,\mu))$ and $v^s_{N_3}(a^s,a^u) = \mathcal{O}(\eta_*^{N_3})$ from (\ref{vBnd}). Our final matching condition is to guarantee that after exactly $m$ iterations the point $(w^s_{N_2}(b^u),w^u_{N_2}(b^u))$ is mapped under the action of $F$ to $(v^s_0(a^u,a^s),v^u_0(a^s,a^u))$. That is, we have to solve  
	\begin{equation}\label{2PulseMatch2}
		F^m((w^s_{N_2}(b^u),w^u_{N_2}(b^u)),\bar{\mu}) - (v^s_0(a^u,a^s),v^u_0(a^s,a^u)) = F^m((\mathcal{O}(\eta_0^{N_2}),\bar{w}^u+b^u),\bar{\mu}) - (a^s,\mathcal{O}(\eta_*^{N_3})) = 0,
	\end{equation}
	where we have applied the facts that $w^s_{N_2}(b^u) = \mathcal{O}(\eta_0^{N_2})$ from (\ref{wBnd}) and $v^u_0(a^s,a^u) = \mathcal{O}(\eta_*^{N_3})$ from (\ref{vBnd}). We note that satisfying (\ref{2PulseMatch1}) and (\ref{2PulseMatch2}) necessarily satisfies all conditions of (\ref{2Pulse}). Indeed, (\ref{2Pulse5}) implies (\ref{2Pulse2}) since $W^s(0,\bar{\mu})$ is an invariant manifold and (\ref{2Pulse3}) follows from reversibility of the solution and the fact that $W^s(0,\bar{\mu}) = \mathcal{R}W^u(0,\bar{\mu})$. Hence, it now remains to solve (\ref{2PulseMatch1}) and (\ref{2PulseMatch2}). 
	
	Let us define the smooth function $H$ that depends on the variables $(b^u,a^s,a^u)$ whose roots correspond to satisfying (\ref{2PulseMatch1}) and (\ref{2PulseMatch2}). Note that $H$ has the expansion
	\begin{equation}\label{2PulseFn}
		H(b^u,a^s,a^u) := \begin{pmatrix}
			G_*(\bar{v}^s + a^u,0,\bar{\mu}) \\
			F^m((0,\bar{w}^u+b^u),\bar{\mu}) - (\bar{v}^s + a^s,0)
		\end{pmatrix} + \mathcal{O}(\eta_0^{N_2} + \eta_*^{N_3}).
	\end{equation}
	We will now work to apply Theorem~\ref{thm:Roots} to our function $H$. Using the notation of Theorem~\ref{thm:Roots}, let us take $x_0 = (0,0,0)$ and define $A$ to be the matrix
	\[
		A = \begin{bmatrix}
			\partial_{v^s}G_*(\bar{v}^s,0,\bar{\mu}) & 0 & 0 \\
			0 & \xi_1 & -1 \\
			0 & \xi_2 & 0
		\end{bmatrix},
	\] 
	where we have defined 
	\[
		(\xi_1,\xi_2) := \nabla_{(w^s,w^u)} F^m((0,\bar{w}^u),\bar{\mu}). 
	\]
	We note that $A$ is invertible because from Lemma~\ref{lem:G*Fn} we have that $\partial_{v^s}G_*(\bar{v}^s,0,\bar{\mu}) \neq 0$ since $(v^s,0)$ represents a transverse intersection between $W^u(0,\bar{\mu})$ and $W^s(u_*,\bar{\mu})$, and similarly, $\xi_2 \neq 0$ since varying $b^u$ in a neighbourhood of zero causes $F^m((0,\bar{w}^u+b^u),\bar{\mu})$ to locally parametrize a connected component of $W^u(0,\bar{\mu})$ transversely intersecting $W^s(u_*,\bar{\mu})$ at $(\bar{v}^s,0) \in \mathcal{I}_*\times\mathcal{I}_*$. Hence, $H(x_0) = \mathcal{O}(\eta_0^{N_2} + \eta_*^{N_3})$ and therefore $\|A^{-1}H(x_0)\| = \mathcal{O}(\eta_0^{N_2} + \eta_*^{N_3})$. Note furthermore that  
	\[
		\|1 - A^{-1}DH(b^u,a^s,a^u)\| = \mathcal{O}(|b^u| + |a^s| + |a^u| + \eta_0^{N_2} + \eta_*^{N_3}),
	\]
	and hence we may apply Theorem~\ref{thm:Roots} with $\kappa = \frac{1}{2}$ and $\rho = \frac{1}{4}$ and $N_2,N_3 \gg 1$ sufficiently large. Therefore, for each $N_2,N_3\gg 1$ sufficiently large there exists a solution, $(b^u,a^s,a^u) = (b^u_*,a^s_*,a^u_*)$, satisfying $H(b^u_*,a^s_*,a^u_*) = 0$ with the property that 
	\[
		\|(b^u_*,a^s_*,a^u_*)\| = \mathcal{O}(\eta_0^{N_2} + \eta_*^{N_3}).
	\] 
	Hence, we have satisfied the matching conditions (\ref{2PulseMatch1})-(\ref{2PulseMatch2}), and from the discussion above we have completed the proof.
\end{proof} %End of proof

Lemma~\ref{lem:Sym2Pulse} is weak in the sense that it only gives existence at single values of $\mu$, but one should note that our matching algorithm employed in Lemma~\ref{lem:Sym2Pulse} can be employed to locally continue the constructed symmetric  2-pulse in $\mu$ as well. Of course this local continuation will only work up to points in $\mu$ where $W^u(0,\mu)$ and $W^s(u_*,\mu)$ no longer intersect transversely.   

To obtain the generic bifurcation structure of homoclinic multi-pulse solutions of (\ref{Map}), we require a genericity hypothesis on the function $G_*$. To begin, by definition of the interval $K_0$, defined in Section~\ref{sec:u_*}, it is possible that multiple iterates of the same heteroclinic orbit lie in this interval. To identify these iterates as the same heteroclinic we introduce the quotient space which for each $\mu \in J$ identifies points on the same trajectory inside $K_0$. The resulting quotient space is then identified as the circle $S^1$ and we denote $q:K_0\times J \to S^1 \times J$ as the associated quotient map which acts as the identity between the $J$ components. Using the quotient map, we impose the following non-degeneracy assumption on the function $G_*$, defined in Lemma~\ref{lem:G*Fn}. 

\begin{hyp}\label{hyp:G*Fn} %Hypothesis: G* function hypothesis
	If $(a,\mu)\in\Gamma_\mathrm{loc}$ is such that $\partial_{v^s}G_*(a,0,\mu) = 0$, then $\partial_{v^s}G_*(\tilde{a},0,\mu) \neq 0$ for all $(\tilde{a},\mu)\in\Gamma_\mathrm{loc}$ with $q(\tilde{a},\mu) \neq q(a,\mu)$.
\end{hyp} %

\begin{figure} %Figure: Hypothesis 4
\centering
\includegraphics[width=0.7\textwidth]{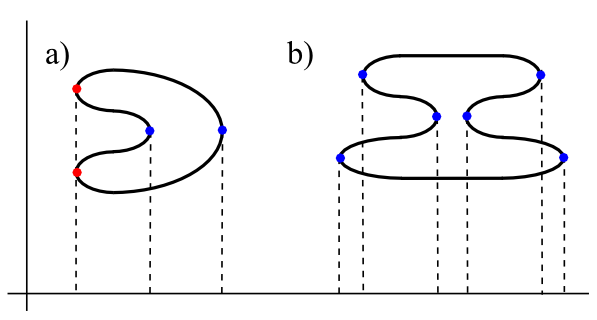}
\caption{From Hypothesis~\ref{hyp:Heteroclinic}, the set $q(\Gamma_\mathrm{loc})$ is a closed curve in $S^1 \times J$. Illustrative examples of two possible curves are given in the figure with $S^1$ representing the horizontal component and $J$ representing the vertical component. The leftmost saddle-nodes (in red) of curve (a) occur for the same value in $S^1$ and therefore violate Hypothesis~\ref{hyp:G*Fn}, making curve (a) inadmissible for the present analysis. Curve (b) is consistent with Hypothesis~\ref{hyp:G*Fn} since all of its saddle-nodes occur for distinct values in $S^1$.}
\label{fig:Hyp4}
\end{figure} 

Hypothesis~\ref{hyp:G*Fn} simply states that for any fixed value of $\mu \in J$ for which there exists at least two distinct heteroclinic orbits belonging to $\Gamma$ with this value of $\mu$, then only one heteroclinic orbit can lie along a quadratic tangency of $W^u(0,\mu)$ and $W^s(u_*,\mu)$. In terms of $q(\Gamma_\mathrm{loc})$ this means that two distinct saddle-nodes cannot occur along this curve for the same value in $S^1$. Illustrative examples of curves that violate and are consistent with Hypothesis~\ref{hyp:G*Fn} are provided in Figure~\ref{fig:Hyp4}. This assumption is of course generic and becomes necessary as we provide the results of Theorem~\ref{thm:MainResult} throughout this subsection and the next. In particular, the following lemma uses Hypothesis~\ref{hyp:G*Fn} to show that symmetric  2-pulses must lie along closed bifurcation curves in $\ell^\infty \times J$.

\begin{lem}\label{lem:Sym2PulseIsola} %Lemma: Symmetric 2-pulses lie on isolas
	Assume Hypothesis~\ref{hyp:Reverser}-\ref{hyp:G*Fn}. The bifurcation curves of each symmetric  2-pulse in $\ell^\infty\times J$ of (\ref{Map}) is a smooth closed curve.
\end{lem}	

\begin{proof}
	Our work in Lemma~\ref{lem:Sym2Pulse} shows that the existence of symmetric  2-pulses is equivalent to satisfying the conditions (\ref{2PulseMatch1}) and (\ref{2PulseMatch2}) for some $(b^u,a^s,a^u)$ at a fixed $\mu \in J$. We begin by noting that in order to satisfy (\ref{2Pulse}) we must have that every homoclinic 2-pulse has $a^s\neq a^u$. Indeed, if $a^s= a^u$ then Lemma~\ref{lem:Shil*Sol} implies that $\{(v^s,v^u)\}_{n = 0}^{N_3}$ is reversible, and reversibility of $\{(w^s,w^u)\}_{n = 0}^{N_2}$ implies that we have constructed a periodic orbit of (\ref{Map}). This is impossible since the matching condition (\ref{2PulseMatch1}) dictates that our orbit lies on the unstable manifold of the fixed point $0$. Hence, every homoclinic 2-pulse has $a^s\neq a^u$. 
	
	Furthermore, we cannot have $q(\bar{v}^s + a^s,\mu) \neq q(\bar{v}^s + a^u,\mu)$ either. The reason for this is that by the definition of the quotient mapping, $q$, this implies that there exists some $p \in \mathbb{Z}$ such that $F^p((a^s,0),\mu) = (a^u,0)$. But then, by the uniqueness of solutions to (\ref{2PulseMatch1})-(\ref{2PulseMatch2}) we would have that the corresponding solution is equivalently generated by replacing $(a^s,0)$ with $F^p((a^s,0),\mu)$, thus arriving at the previously discussed scenario. Hence, $q(\bar{v}^s +a^s,\mu) \neq q(\bar{v}^s +a^u,\mu)$ for all symmetric  2-pulses of (\ref{Map}).   
	
	Now, following a curve of homoclinic 2-pulses is equivalent to following a curve in the variables $(b^u,a^s,a^u,\mu)$ with $(b^u,a^s,a^u)$ satisfying (\ref{2PulseMatch1}) and (\ref{2PulseMatch2}) for each $\mu \in J$. From our discussion above, we must have $q(\bar{v}^s +a^s,\mu) \neq q(\bar{v}^s +a^u,\mu)$ for each $\mu$ along the curve, and moreover, since our solution is a homoclinic orbit, it follows that 
	\[
		(\bar{v}^s+a^s,v^u_0(a^s,a^u)) \in W^u(0,\mu)
	\] 
	for all $a^s$ along the curve. Since $v^u_0(a^s,a^u)$ is exponentially small, it follows from Lemma~\ref{lem:G*Fn} that 
	\[
		G_*(\bar{v}^s+a^s,v^u_0(a^s,a^u),\mu) = 0
	\]    
	at each point along this curve. But, since $q(\bar{v}^s +a^s,\mu) \neq q(\bar{v}^s +a^u,\mu)$ at every point on this curve, the arguments of \cite[Lemma~4.5]{Bramburger} show that solutions which satisfy both
	\[
		\begin{split}
			G_*(\bar{v}^s+a^s,v^u_0(a^s,a^u),\mu) &= 0, \\
			G_*(\bar{v}^s+a^u,v^s_{N_3}(a^s,a^u),\mu) &= 0,	
		\end{split}
	\]  
	with $q(\bar{v}^s +a^s,\mu) \neq q(\bar{v}^s +a^u,\mu)$ must form a closed curve in $(a^s,a^u,\mu)$-space. Then, uniqueness of orbits of $F$ implies that this closed curve in $(a^s,a^u,\mu)$-space extends to a closed curve in $(b^u,a^s,a^u,\mu)$-space. This completes the proof.
\end{proof} %End of proof

\subsection{Asymmetric 2-Pulses}\label{subsec:Asym2} %Subsection: Asymmetric 2-Pulses -----------------------------------------------------------------------------------------------------------------

We now discuss the existence of asymmetric  2-pulse solutions of (\ref{Map}). Here, by definition, an asymmetric  2-pulse $u = \{u_n\}_{n\in\mathbb{Z}}$ satisfies
\begin{subequations}\label{Asym2Pulse}
	\begin{equation}\label{Asym2Pulse1}
		u_{-N_1-m} \in \Sigma_\mathrm{in}\cap W^u(0,\mu) 	
	\end{equation}
	\begin{equation}\label{Asym2Pulse2}
		u_n \in \mathcal{I}_*\times\mathcal{I}_*\ {\rm for}\ n\in\{-N_1-m,\dots,-m\}	
	\end{equation}
	\begin{equation}\label{Asym2Pulse3}
		u_{-m} \in \Sigma_\mathrm{out}\cap W^s(0,\mu) 	
	\end{equation}
	\begin{equation}\label{Asym2Pulse4}
		u_n \in \mathcal{I}_0\times\mathcal{I}_0\ {\rm for}\ n\in\{0,\dots,N_2\}	
	\end{equation}
	\begin{equation}\label{Asym2Pulse5}
		u_{N_2+m} \in \Sigma_\mathrm{in}\cap W^u(0,\mu) 	
	\end{equation}
	\begin{equation}\label{Asym2Pulse6}
		u_n \in \mathcal{I}_*\times\mathcal{I}_*\ {\rm for}\ n\in\{N_2+m,\dots,N_2 + N_3+m\}	
	\end{equation}
	\begin{equation}\label{Asym2Pulse7}
		u_{N_2+N_3+m} \in \Sigma_\mathrm{out}\cap W^s(0,\mu) 	
	\end{equation}
\end{subequations}  
for sufficiently large $N_1,N_2,N_3 \gg 1$, and we recall that $m \geq 0$ is the integer defined in Lemma~\ref{lem:mdefn}. We note that in the case $N_1 = N_3$, our definition becomes nearly identical to that of the symmetric  2-pulses, with the exception of the symmetry condition in (\ref{2Pulse1}). In this section we will show that there exists homoclinic 2-pulses with $N_1 = N_3$ but do not satisfy $\mathcal{R}u = u$ which bifurcate from the curve of symmetric  2-pulses. Furthermore, our work in this section covers to the much more general case of $N_1 \neq N_3$, which implies that the homoclinic 2-pulse cannot be symmetric. We begin by providing the following existence result, akin to Lemma~\ref{lem:Sym2Pulse} for symmetric 2-pulses.

\begin{lem}\label{lem:Asym2Pulse}%Lemma: Existence of asymmetric 2-pulses
	Assume Hypothesis~\ref{hyp:Reverser}-\ref{hyp:Heteroclinic} and that at $\mu = \bar{\mu}$ the manifolds $W^u(0,\bar{\mu})$ and $W^s(u_*,\bar{\mu})$ intersect transversely at the point $(\bar{v}^s,0,\bar{\mu}) \in \Gamma_\mathrm{loc}$. Then, there exists $M_2^\mathrm{a} > 0$ such that for all $N_1,N_2,N_3 > M_2^\mathrm{a}$ the map (\ref{Map}) evaluated at $\mu=\bar{\mu}$ has a homoclinic 2-pulse solution satisfying (\ref{Asym2Pulse}).
\end{lem}

\begin{proof}
	The proof of this lemma is very similar to that of Lemma~\ref{lem:Asym2Pulse}, and therefore we will only proceed until the methods become equivalent. As before, we take $\mu =\bar{\mu}$ fixed throughout, and again taking $(\bar{v}^s,0,\bar{\mu}) \in \Gamma_\mathrm{loc}$, we use Lemma~\ref{lem:mdefn} to define $\bar{w}^u$ so that 
	\[
		(\bar{w}^u,0) = F^{-m}((\bar{v}^s,0),\bar{\mu}) \in (\mathcal{I}_0\times\mathcal{I}_0)\cap W^u(0,\bar{\mu}) \cap W^s(u_*,\bar{\mu}).
	\]  
	Now, for any $a^s_1,a^u_1$ sufficiently small and $N_1 \geq 1$, Lemma~\ref{lem:Shil*Sol} gives the existence of a solution $\{v^s_{1,n},v^u_{1,n}\}_{n=0}^{N_1} \subset \mathcal{I}_*\times\mathcal{I}_*$ satisfying 
	\[
		v^s_{1,0} = \bar{v}^s + a^s_1, \quad v^u_{1,N_1} = \bar{v}^s + a^u_1,
	\]
	where $v^j_{1,n} = v^j_{1,n}(a^s_1,a^u_1)$ depends smoothly on $a^s_1,a^u_1\in\mathcal{I}_*$ for $j = s,u$. Similarly, for any $a^s_3,a^u_3 \in \mathcal{I}_*$ and $N_3 \geq 1$ we can use Lemma~\ref{lem:Shil*Sol} again to obtain a solution $\{v^s_{3,n},v^u_{3,n}\}_{n=0}^{N_3} \subset \mathcal{I}_*\times\mathcal{I}_*$ with the same properties as $\{v^s_{1,n},v^u_{1,n}\}_{n=0}^{N_1}$. 
	
	Then, using Lemma~\ref{lem:Shil0Sol} we have that for every $b^s,b^u\in\mathcal{I}_0$ and $N_2 \geq 1$, there exists a solution $\{w^s_n,w^u_n\}_{n=0}^{N_2} \subset \mathcal{I}_0\times\mathcal{I}_0$ satisfying 
	\[
		w^s_0 = \bar{w}^u + b^s, \quad w^u_{N_2} = \bar{w}^u + b^u,
	\]
	where $w^j_n = w^j_n(b^s,b^u)$ depends smoothly on $b^s,b^u\in\mathcal{I}_0$ for $j = s,u$. In a similar fashion to the proof of Lemma~\ref{lem:Sym2Pulse}, satisfying (\ref{Asym2Pulse}) now becomes equivalent to solving the following matching conditions:
	\begin{gather*}
		G_*(v^s_{1,0}(a^s_1,a^u_1),v^u_{1,0}(a^s_1,a^u_1),\bar{\mu}) = 0, \\
		F^{-m}((w^s_0(b^s,b^u),w^u_0(b^s,b^u)),\bar{\mu}) - (v^s_{1,N_1}(a^s_1,a^u_1),v^u_{1,N_1}(a^s_1,a^u_1)), \\
		F^{m}((w^s_{N_2}(b^s,b^u),w^u_{N_2}(b^s,b^u)),\bar{\mu}) - (v^s_{3,0}(a^s_3,a^u_3),v^u_{3,0}(a^s_3,a^u_3)), \\
		G_*(v^u_{3,N_1}(a^s_3,a^u_3),v^s_{3,N_1}(a^s_3,a^u_3),\bar{\mu}) = 0.
	\end{gather*}
	Based upon the criteria (\ref{Asym2Pulse}) it is easy to check that these conditions do indeed lead to a homoclinic 2-pulse solution of (\ref{Map}). Furthermore, we may gather these matching conditions to define a smooth function $H^a$ that depends on the variables $(a^s_1,a^u_1,b^s,b^u,a^s_3,a^u_3)$ so that the roots of $H^a$ are exactly solutions to our matching conditions. Using the asymptotic expansions (\ref{Shil0Bnds}) and (\ref{Shil*Bnds}) we find that $H^a$ has the asymptotic expansion
	\begin{equation}\label{H^a}
		H^a(a^s_1,a^u_1,b^s,b^u,a^s_3,a^u_3) = \begin{bmatrix}
			G_*(\bar{v}^s + a^s_1,0,\bar{\mu}) \\
			F^{-m}((\bar{w}^s+b^s,0),\bar{\mu}) - (0,\bar{v}^s + a^u_1) \\
			F^{m}((0,\bar{w}^s+b^s),\bar{\mu}) - (\bar{v}^s + a^s_3,0) \\
			G_*(\bar{v}^s + a^u_3,0,\bar{\mu})
		\end{bmatrix} + \mathcal{O}(\eta_*^{N_1} + \eta_0^{N_2} + \eta_*^{N_3}).
	\end{equation}
	We can now see that $H(0,0,0,0,0,0) = \mathcal{O}(\eta_*^{N_1} + \eta_0^{N_2} + \eta_*^{N_3})$, and obtaining roots of $H^a$ is simply an application of Theorem~\ref{thm:Roots}, which is handled in a nearly identical way to the symmetric case of Lemma~\ref{lem:Sym2Pulse}. Therefore the rest of the proof is omitted. 
\end{proof} %End of proof

We follow the existence proof of Lemma~\ref{lem:Asym2Pulse} with the following bifurcation result for the case that $N_1 \neq N_3$. 

\begin{lem}\label{lem:Asym2PulseIsola} %Lemma: Asymmetric 2-pulses lie on isolas
	Assume Hypothesis~\ref{hyp:Reverser}-\ref{hyp:G*Fn}. The bifurcation curve of each asymmetric  2-pulse of (\ref{Map}) with $N_1 \neq N_3$ in $\ell^\infty\times J$ is a smooth closed curve. 
\end{lem}	

\begin{proof}
	This proof is the same as that of Lemma~\ref{lem:Sym2PulseIsola}.
\end{proof} %End of proof

We note again that in the case $N_1 = N_3$ it could be the case that Lemma~\ref{lem:Asym2Pulse} simply provides the existence of the symmetric  2-pulses only. The following lemma shows that that is not the case. 

\begin{lem}\label{lem:Sym2Pitchforks} %Lemma: Asymmetric 2-pulses bifurcate from symmetric 2-pulses
	Assume Hypotheses~\ref{hyp:Reverser}-\ref{hyp:G*Fn} are met. Then there exists $\eta \in (0,1)$ such that for each $N_2,N_3 \geq 1$ sufficiently large and $\mu_\mathrm{sn} \in \mathring{J}$, the location of a saddle-node bifurcation on the symmetric  2-pulse curve associated to $N_2,N_3$, precisely two branches of asymmetric  2-pulses (mapped into each other by $\mathcal{R}$) bifurcate from the symmetric  2-pulse associated to the integers $N_2,N_3$ at $\mu = \mu_\mathrm{pf}$ with $|\mu_\mathrm{pf} - \mu_\mathrm{sn}| = \mathcal{O}(\eta^{\min\{N_2,N_3\}})$. 
\end{lem}

\begin{proof}
	Let us fix $N_2 = N_3$ sufficiently large to guarantee the existence a symmetric  2-pulse by Lemmas~\ref{lem:Sym2Pulse} and \ref{lem:Sym2PulseIsola}. Using the function $H^a$ defined in (\ref{H^a}) we note that the corresponding homoclinic 2-pulse is symmetric if, and only if, $b^s = b^u$. A consequence of this fact is that if the corresponding 2-pulse is symmetric, then necessarily we have $a^s_1 = a^u_3$ and $a^u_1 = a^s_3$. We again note that we argued in the proof of Lemma~\ref{lem:Sym2PulseIsola} that $q(\bar{v}^s+a^s_1,\mu) \neq q(\bar{v}^s+a^u_1,\mu)$ at any point $\mu \in J$ for which a symmetric  2-pulse exists. Furthermore, one can see from the proof of Lemma~\ref{lem:Sym2PulseIsola} that a saddle-node bifurcation along the closed curve associated to the symmetric  2-pulse with $N_2,N_3$ sufficiently large occurs at $\mu_\mathrm{sn} = \mu_0 + \mathcal{O}(\eta^{\min\{N_2,N_3\}})$, for some $\eta \in (0,1)$, where $\mu_0 \in J$ is such that $W^u(0,\mu_0)$ and $W^s(u_*,\mu_0)$ intersect in a quadratic tangency. Now, Hypothesis~\ref{hyp:G*Fn} implies that 
	\[
		\begin{split}
			&\partial_{v^s}G_*(v^s_{1,0}(a^s_1,a^u_1),v^u_{1,0}(a^s_1,a^u_1),\mu_\mathrm{sn}) = 0, \\ 
			&\partial_{v^s}G_*(v^u_{1,N_2}(a^s_1,a^u_1),v^s_{1,N_2}(a^s_1,a^u_1),\mu_\mathrm{sn}) \neq 0,  
		\end{split}
	\]
	or  
	\[
		\begin{split}
			&\partial_{v^s}G_*(v^s_{1,0}(a^s_1,a^u_1),v^u_{1,0}(a^s_1,a^u_1),\mu_\mathrm{sn}) \neq 0, \\ 
			&\partial_{v^s}G_*(v^u_{1,N_2}(a^s_1,a^u_1),v^s_{1,N_2}(a^s_1,a^u_1),\mu_\mathrm{sn}) = 0.  
		\end{split}
	\]
	We will focus on the former case for the duration of this proof since the latter case is handled in exactly the same way. 
	
	Let us denote the values of $a^s_1 = a^u_3$ of the symmetric  2-pulse at $\mu = \mu_\mathrm{sn}$ simply by $a_\mathrm{sn}$, so that now we may apply a Lyapunov-Schmidt reduction to $H^a$ in a neighbourhood of $(a^s_1,a^u_3,\mu) = (a_\mathrm{sn},a_\mathrm{sn},\mu_\mathrm{sn})$ to write $(a^u_1,b^s,b^u,a^s_3)$ as smooth functions of $(a^s_1,a^u_3,\mu)$ for $\mu$. Furthermore, we obtain the bifurcation functions
	\[
		\begin{split}
		h_1(a^s_1,a^u_3,\mu) &= G_*(\bar{v}^s+a^s_1,0,\mu) + \mathcal{O}(\eta_*^{N_2} + \eta_0^{N_3}), \\
		h_2(a^s_1,a^u_3,\mu) &= G_*(\bar{v}^s+a^u_3,0,\mu) + \mathcal{O}(\eta_*^{N_2} + \eta_0^{N_3}),
		\end{split}
	\]   
	which now remain to be solved to obtain bifurcating asymmetric  2-pulses.
	
	Let us now introduce the function $\mathcal{H}(a^s,a^u,\mu)$ given by
	\[
		\begin{split}
		\mathcal{H}(a^s,a^u,\mu) &= \begin{bmatrix} 
			\mathcal{H}_1(a^s_1,a^u_3,\mu) \\
			\mathcal{H}_2(a^s_1,a^u_3,\mu)
		\end{bmatrix} \\
		:&= \begin{bmatrix}
			h_1(a^s_1,a^u_3,\mu) + h_2(a^s_1,a^u_3,\mu) \\
			h_1(a^s_1,a^u_3,\mu) - h_2(a^s_1,a^u_3,\mu)
		\end{bmatrix} \\
		 &= \begin{bmatrix}
			G_*(\bar{v}^s+a^s_1,0,\mu) + G_*(\bar{v}^s+a^u_3,0,\mu) \\
			G_*(\bar{v}^s+a^s_1,0,\mu) - G_*(\bar{v}^s+a^u_3,0,\mu)
		\end{bmatrix} + \mathcal{O}(\eta_*^{N_2} + \eta_0^{N_3})
		\end{split}
	\] 
	For notational convenience we will drop the $\mathcal{O}(\eta_*^{N_2} + \eta_0^{N_3})$ terms when analyzing the function since Theorem~\ref{thm:Roots} implies that a root of 
	\[
		\begin{bmatrix}
			h_1(a^s_1,a^u_3,\mu) + h_2(a^s_1,a^u_3,\mu) \\
			h_1(a^s_1,a^u_3,\mu) - h_2(a^s_1,a^u_3,\mu)
		\end{bmatrix}	
	\] 
	can be extended uniquely to a root of $\mathcal{H}$. Now, notice that $\partial_\mu\mathcal{H}_1(a_\mathrm{sn},a_\mathrm{sn},\mu_\mathrm{sn}) = 2\partial_\mu G_*(a_\mathrm{sn},0,\mu_\mathrm{sn}) \neq 0$, so that the implicit function theorem implies that we may solve $\mathcal{H}_1(a^s,a^u,\mu)$ near $(a_\mathrm{sn},a_\mathrm{sn},\mu_\mathrm{sn})$ uniquely for $\mu = \mu_*(a^s_1,a^u_3)$ as a function of $(a^s_1,a^u_3)$ with $\mu_*(a_\mathrm{sn},a_\mathrm{sn}) = \mu_\mathrm{sn}$. The function $\mu_*$ further has the property that $\mu_*(a^s_1,a^u_3) = \mu_*(a^u_3,a^s_1)$ for all $(a^s_1,a^u_3)$. 
	
	We now obtain roots of $\mathcal{H}_2$. Putting $\mu_*$ into $\mathcal{H}_2$ we note that 
	\begin{equation}\label{H2Symmetry}
		\mathcal{H}_2(a^s_1,a^u_3,\mu_*(a^s_1,a^u_3)) = -\mathcal{H}_2(a^u_3,a^s_1,\mu_*(a^u_3,a^s_1))	
	\end{equation}
	for all $(a^s_1,a^u_3)$ by simply using the form of $\mathcal{H}_2$ and the symmetry of $\mu_*$. Expanding $\mathcal{H}_2$ as a Taylor series about $(a^s_1,a^u_3) = (a_\mathrm{sn},a_\mathrm{sn})$ gives 
	\[
		\begin{split}
		\mathcal{H}_2(a^s_1,a^u_3,\mu_*(a^s_1,a^u_3)) &= \partial^2_{v^s}G_*(a_\mathrm{sn},0,\mu_\mathrm{sn})(a^s_1 - a_\mathrm{sn})^2 - \partial^2_{v^s}G_*(a_\mathrm{sn},0,\mu_\mathrm{sn})(a^u_3 - a_\mathrm{sn})^2 + \mathcal{O}(|a^s_1 - a_\mathrm{sn}|^3 + |a^u_3 - a_\mathrm{sn}|^3)  \\
		&= (a^s_1 - a^u_3)\bigg(2\partial^2_{v^s}G_*(a_\mathrm{sn},0,\mu_\mathrm{sn})(a^s_1 + a^u_3 - 2a_\mathrm{sn}) + \mathcal{O}(|a^s_1 - a_\mathrm{sn}|^2 + |a^u_3 - a_\mathrm{sn}|^2)\bigg),
		\end{split}
	\]	
	where we are able to factor $(a^u_3 - a^s_1)$ out from all terms due to the symmetry (\ref{H2Symmetry}). Since $\mu_\mathrm{sn}$ is $\mathcal{O}(\eta^{\min\{N_2,N_3\}})$-close to a value of $\mu \in J$ where $W^u(0,\mu_0)$ and $W^s(u_*,\mu_0)$ intersect in a quadratic tangency, it follows that $2\partial^2_{v^s}G_*(a,0,\mu_\mathrm{sn}) \neq 0$, and hence we may use the implicit function theorem to solve the second factor of our expansion of $\mathcal{H}_2(a^s_1,a^u_3,\mu_*(a^s_1,a^u_3))$ for $a^s_1$ as a function of $a^u_3$ in a neighbourhood of $a^u_3 = a_\mathrm{sn}$, which shows that an asymmetric solution bifurcates from the symmetric solution given by $a^s_1 = a^u_3 = a_\mathrm{sn}$. The assertion that these bifurcating solutions are mapped into each other by $\mathcal{R}$ follows from the fact that if $u$ is a solution of (\ref{Map}), then so is $\mathcal{R}u$ and uniqueness of roots of $H^a$.
\end{proof} %End of proof

We now conclude this section with the following lemmas which extend Lemma~\ref{lem:Asym2PulseIsola}. The proofs will be omitted since the first is nearly identical to that of Lemmas~\ref{lem:Sym2PulseIsola} and \ref{lem:Asym2PulseIsola} and the second is identical to \cite[Lemma~4.6]{Bramburger}.

\begin{lem}\label{lem:Asym2PulseIsola2} %Lemma: Bifurcation curves of asymmetric 2-pulses with N_1 = N_3
	Assume Hypothesis~\ref{hyp:Reverser}-\ref{hyp:G*Fn}. The bifurcation curve of each asymmetric  2-pulse of (\ref{Map}) with $N_1 = N_3$ in $\ell^\infty\times J$ is either a smooth closed curve or a smooth curve with boundaries given by the pitchfork bifurcations described in Lemma~\ref{lem:Sym2Pitchforks}.
\end{lem}

\begin{lem}\label{lem:Asym2PulseIsola3} %Lemma: Pitchforks originate and terminate at opposite curvature
	Assume Hypothesis~\ref{hyp:Reverser}-\ref{hyp:G*Fn}. The branches of asymmetric  2-pulses described in Lemma~\ref{lem:Sym2Pitchforks} begin and end at pitchfork bifurcations near saddle-node bifurcations of symmetric  2-pulses of opposite curvature.
\end{lem}

We now conclude this section by remarking that Lemma~\ref{lem:Asym2PulseIsola} implies that asymmetric 2-pulses with $N_1 \neq N_3$ can only exhibit saddle-node bifurcations. More precisely, a result analogous to Lemma~\ref{lem:Asym2PulseIsola2} does not exist for the case $N_1 \neq N_3$. Hence, on the closed bifurcation curve of an asymmetric 2-pulse with $N_1\neq N_3$ we cannot have a bifurcation to another asymmetric 2-pulse, and hence these asymmetric 2-pulses are isolated in $\ell^\infty\times J$.

\subsection{Extension to $k$-Pulses}\label{subsec:kPulse} %Subsection: k-Pulses ----------------------------------------------------------------------------------------------------------------------------------

Here we now discuss how the previous work for $2$-pulses can be extended to $k$-pulses for arbitrary $k \geq 2$. We will demonstrate that having proven all details of Theorem~\ref{thm:MainResult} in the case $k =2$, the cases $k \geq 3$ follow in a straightforward way. For this reason we refrain from proving the remaining cases in full detail, but simply describe the problem setup and demonstrate that the methods are identical to the methods for the case $k = 2$. 

We begin by fixing some $k \geq 2$ and again take $\bar{v}^s \in\mathring{\mathcal{I}}_*$ and $\bar{w} \in \mathring{\mathcal{I}}_0$ as given in Lemmas~\ref{lem:Sym2Pulse} and \ref{lem:Asym2Pulse}. Now, consider a sequence of arbitrary positive integers $N_1,N_2,\dots,N_{2k-1}$. Then, for $N_i$ with $i$ odd we take $a^s_i,a^u_i$ sufficiently small and use the results of Lemma~\ref{lem:Shil*Sol} to obtain the solution $\{v^s_{i,n},v^u_{i,n}\}_{n=0}^{N_i} \subset \mathcal{I}_*\times\mathcal{I}_*$ of (\ref{Shil*}) satisfying 
\begin{equation}\label{kPulse1}
	v^s_{i,0} = \bar{v}^s + a^s_i, \quad v^u_{i,N_1} = \bar{v}^s + a^u_i,
\end{equation}
where $v^j_{i,n} = v^j_{i,n}(a^s_i,a^u_i)$ depends smoothly on $a^s_i,a^u_i$ in a neighbourhood of $(0,0)$ for $j = s,u$ and all $n \in\{0,\dots,N_i\}$. Similarly, for $N_i$ with $i$ even we take $b^s_i,b^u_i$ sufficiently small and use the results of Lemma~\ref{lem:Shil0Sol} to obtain the solution $\{w^s_{i,n},w^u_{i,n}\}_{n=0}^{N_i} \subset \mathcal{I}_0\times\mathcal{I}_0$ of (\ref{Shil0}) satisfying 
\begin{equation}\label{kPulse2}
	w^s_{i,0} = \bar{w} + b^s_i, \quad v^u_{i,N_1} = \bar{w} + b^u_i,
\end{equation}
where $w^j_{i,n} = w^j_{i,n}(b^s_i,b^u_i)$ depends smoothly on $b^s_i,b^u_i$ in a neighbourhood of $(0,0)$ for $j = s,u$ and all $n \in\{0,\dots,N_i\}$. We note that the first index denotes which $N_i$ the solution pertains to and the second index relates to the iterations under the map $F$.

Having now the solutions (\ref{kPulse1}) and (\ref{kPulse2}), we seek to choose the variables 
\[
	(a^s_1,a^u_1,b^s_2,b^u_2,\dots,b^s_{2k-2},b^u_{2k-2},a^s_{2k-1},a^u_{2k-1})
\] 
appropriately to patch together the $2k-1$ solution fragments, as well as guarantee that the solution is indeed a homoclinic orbit. To do this, we first use Lemma~\ref{lem:G*Fn} to define the matching conditions 
\begin{equation}\label{kPulse3}
	G_*(\bar{v}^s + a^s_1,v^u_{1,0}(a^s_1,a^u_1),\mu) = 0	
\end{equation}
and
\begin{equation}\label{kPulse4}
	G_*(\bar{v}^s + a^u_{2k-1},v^u_{2k-1,N_{2k-1}}(a^s_{2k-1},a^u_{2k-1}),\mu) = 0,	
\end{equation}
which when satisfied for some $\mu \in J$ gives that the solution lies in $W^s(0,\mu)\cap W^u(0,\mu)$. Hence, satisfying conditions (\ref{kPulse3}) and (\ref{kPulse4}) guarantees that the solution is indeed a homoclinic orbit. Then, as in the Lemmas~\ref{lem:Sym2Pulse} and \ref{lem:Asym2Pulse}, the remaining matching conditions stitch together the successive solutions by requiring that for some fixed $\mu \in J$ we satisfy 
\begin{equation}\label{kPulse5}
	F^{-m}((w^s_{i+1,0}(b^s_{i+1},b^u_{i+1}),w^u_{i+1,0}(b^s_{i+1},b^u_{i+1})),\mu) - (v^s_{i,N_i}(a^s_i,a^u_i),v^u_{i,N_i}(a^s_i,a^u_i)))=0
\end{equation} 
for all $i \in \{1,3,\dots,2k-3\}$ and
\begin{equation}\label{kPulse6}
	F^m((w^s_{i,N_i}(b^s_i,b^u_i),w^u_{i,N_i}(b^s_i,b^u_i))),\mu) - (v^s_{i+1,0}(a^s_{i+1},a^u_{i+1}),v^u_{i+1,0}(a^s_{i+1},a^u_{i+1}))=0
\end{equation}
for all $i \in \{2,4,\dots,2k-2\}$. We recall that $m \geq 1$ is the constant given in Lemma~\ref{lem:mdefn} which approximately describes the number of iterates to move from $\mathcal{I}_0\times\mathcal{I}_0$ to $\mathcal{I}_*\times\mathcal{I}_*$. It should be noted that in the case $k = 2$ the matching conditions (\ref{kPulse3})-(\ref{kPulse6}) are exactly those used to obtain $2$-pulse solutions in Lemma~\ref{lem:Asym2Pulse}.

Taking some $\bar{\mu}\in J$ such that $W^u(0,\bar{\mu})$ and $W^s(u_*,\bar{\mu})$ intersect transversely, the conditions (\ref{kPulse3})-(\ref{kPulse6}) can be satisfied for arbitrary $k \geq 3$ in a nearly identical process to the case $k = 2$ handled above. Furthermore, searching for symmetric solutions requires the added condition that $b^s_k = b^u_k$, which can be used to reduce the number of equations in (\ref{kPulse3})-(\ref{kPulse6}) to be solved, much like the case $k = 2$ handled in Lemma~\ref{lem:Sym2Pulse}. Upon performing this matching to obtain $k$-pulses when $W^u(0,\bar{\mu})$ and $W^s(u_*,\bar{\mu})$ intersect transversely, we may prove the remaining statements of Theorem~\ref{thm:MainResult} by following as in the case $k =2$ detailed above.

\section{Discussion}\label{sec:Discussion} %Section: Discussion ---------------------------------------------------------------------------------------------------------------------------------

In this paper, we analyzed the existence and generic bifurcation structure of multi-pulse solutions to lattice dynamical systems posed on $\mathbb{Z}$. These results extended previous investigations into single-pulses \cite{Bramburger} and go far beyond what is known for multi-pulses in the continuous spatial setting \cite{2Pulse}. To obtain our results we analyzed homoclinic orbits of two-dimensional reversible maps which enter and leave a neighbourhood of another fixed point multiple times. Similar to the case of single-pulses, the key to demonstrating the existence and bifurcation structure of these solutions is to understand the global bifurcation structure of back solutions, which manifest themselves as heteroclinic orbits of the associated map. Importantly, the bifurcation structure of the back solutions generically dictates whether single-pulse solutions will snake or not, but we saw in this work that regardless of the behaviour of single-pulse solutions, generically all multi-pulse solutions lie along isolas. Finally, we used a spatially-discrete Nagumo equation to demonstrate our results numerically as well as analytically confirm the hypotheses required to apply our main results. 

As in many investigations into localized structures in the continuous spatial setting, here we assumed that the associated spatial dynamical system has a reversible structure. We saw in the example of the Nagumo equation that the reversible structure of the spatial dynamical system follows from the symmetry of the coupling function. Moreover, reversibility greatly simplifies our analysis since front and back solutions of the lattice dynamical system must lie in one-to-one correspondence with each other, thus allowing one to formulate Hypothesis~\ref{hyp:Heteroclinic} in terms of only back solutions. In this way, reversibility greatly reduces the complexity of presenting the results and providing the existence proofs. Reversible symmetry also lends itself to the existence of symmetric solutions, which have a richer bifurcation structure than asymmetric solutions due to the symmetry-breaking pitchfork bifurcations occurring near the saddle-nodes along their bifurcation curves. Therefore, in the absence of reversible symmetry one would be required to formulate hypotheses akin to Hypothesis~\ref{hyp:Heteroclinic} for both front and back solutions, and then appropriately follow much of the analysis in this manuscript to obtain multi-pulse solutions of an associated lattice dynamical system. 

Another approach to understanding the effect of losing reversibility would be to consider non-reversible perturbations, for which we conjecture that the isolas of asymmetric multi-pulses described herein perturb regularly and remain distinct from each other. The case of the symmetric multi-pulses becomes more delicate. That is, it has been shown that non-reversible perturbations can cause the snaking bifurcation curves of single-pulse solutions in the lattice setting to degenerate into isolas \cite{Yulin2}. This process is partially understood by the fact that the pitchfork bifurcations do not survive the symmetry-breaking perturbations, thus causing the bifurcation curves to fragment into isolas. Therefore, it should be expected that a similar fragmentation takes place along the isolas of symmetric multi-pulse solutions since we again have pitchfork bifurcations that should not be assumed to be robust under a symmetry-breaking perturbation. Nonetheless, based on the arguments in the previous paragraph, we expect that all multi-pulse solutions perturb regularly and remain as steady-state solutions to the lattice system, with only their organization in parameter space being effected by such a symmetry-breaking perturbation. 

Finally, a future direction for this research is to understand the existence, stability, and bifurcation structure of localized solutions with multiple distinct regions of localization (hereby multi-localized solution) on two-dimensional lattices. In Figure~\ref{fig:2D} we present an isola of multi-localized solutions to the Nagumo lattice dynamical system on a two-dimension lattice, given by
\begin{equation}\label{LDS2D}
	\dot{U}_{n,m} = d(U_{n+1,m} + U_{n-1,m} + U_{n,m+1} + U_{n,m-1} - 4U_{n,m}) + U_{n,m}(U_{n,m} - \mu)(1-U_{n,m}), \quad (n,m)\in\mathbb{Z}^2.
\end{equation}
This numerical computation and others lead one to conjecture that similar to the one-dimensional lattice case, multi-localized solutions of (\ref{LDS2D}) all lie along isolas in parameter space. The major barrier to proving this is that the spatial dynamics methods of this manuscript cannot be extended to demonstrate the existence of localized solutions to systems on two-dimensional lattices. An alternative method to proving the existence and bifurcation structure of multi-localized solutions to (\ref{LDS2D}) is to continue solutions up from the singular parameter value $d = 0$, as was done to confirm Hypothesis~\ref{hyp:Heteroclinic} in Proposition~\ref{prop:Heteroclinic}. This method has recently been applied to system (\ref{LDS2D}) to demonstrate that simple square patterns exhibit snaking bifurcation curves for $0 < d \ll 1$ \cite{Bramburger2}. The drawback to this method is that distinct multi-localized solutions must be analyzed individually, requiring an a priori knowledge of the bifurcation structure. Hence, it remains to find an efficient method for examining localized solutions to lattice dynamical systems posed on two-dimensional lattices. 

\begin{figure} %Figure: 2D multi-pulses
\centering
\includegraphics[width=0.9\textwidth]{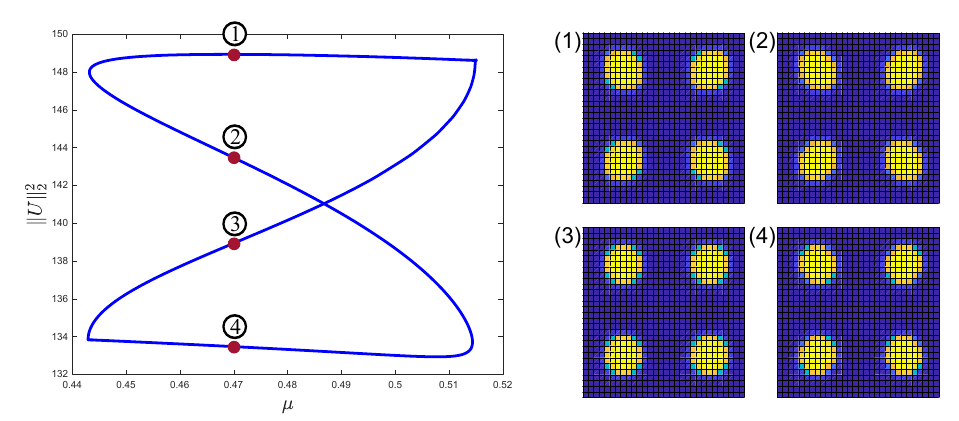}
\caption{An isola of a $D_4$-symmetric multi-localized steady-state of (\ref{LDS2D}) with $d = 0.05$. Contour plots of sample profiles are provided at $\mu = 0.47$.}
\label{fig:2D}
\end{figure} 

\paragraph{Acknowledgements.}
This work was supported by an NSERC PDF. The author would also like to thank Bj\"orn Sandstede for many conversations which led to this work.

%%%%%%%%%%%%%%%%%%%%%%%%%%%%%%%%%%%%%%%%%%%%%%%%%%%%%%%%%%%%%%%%%%%%%%%%%%%%

\end{document}